\documentclass{amsart}
\usepackage{amsfonts}

\newtheorem{theorem}{Theorem}
\theoremstyle{plain}

\newtheorem{corollary}{Corollary}

\newtheorem{lemma}{Lemma}

\numberwithin{equation}{section}
\input{tcilatex}

\begin{document}
\title{Density functions in high-dimensional basket options}
\author{Alexander Kushpel}
\address{Department of Mathematics, University of Leicester, University
Road, Leicester, LE1 7RH}
\email{ak412@le.ac.uk}
\date{August 2013.}

\begin{abstract}
We consider an important class of derivative contracts written on multiple
assets (so-called spread options) which are traded on a wide range of
financial markets. The present paper introduces a new approximation method
of density functions arising in high-dimensional basket options which is
based on applications of generalised Nyquist-Whitakker-Kotel'nikov-Shannon
theorem we established. It is shown that the method of approximation we
propose has an exponential rate of convergence in various situations.
\end{abstract}

\maketitle

\section{introduction}

\label{sec1}

Consider a frictionless market with no arbitrage opportunities with a
constant riskless interest rate $r>0$. Let $S_{1,t}$ and $S_{2,t},t\geq 0$,
be two asset price processes. Consider a European call option on the price
spread $S_{1,T}-S_{2,T}$. The common spread option with maturity $T>0$ and
strike $K\geq 0$ is the contract that pays $\left( S_{1,T}-S_{2,T}-K\right)
_{+}$ at time $T$, where $\left( a\right) _{+}:=\max \left\{ a,0\right\} $.
There is a wide range of such options traded across different sectors of
financial markets. For instance, the crack spread and crush spread options
in the commodity markets \cite{Mbanefo}, \cite{Shimko}, credit spread
options in the fixed income markets, index spread options in the equity
markets \cite{Duan} and the spark (fuel/electricity) spread options in the
energy markets \cite{Deng}, \cite{Pilipovic}.

Assuming the existence of a risk-neutral equivalent martingale measure we
get the following pricing formula for the call value at time $0$.%
\begin{equation}
V=V\left( S_{1,0},S_{2,0},T\right) =e^{-rT}\mathbb{E}\left[ \left(
S_{1,T}-S_{2,T}-K\right) _{+}\right] ,  \label{vvv}
\end{equation}%
where the expectation is taken with respect to the equivalent martingale
measure. There is an extensive literature on spread options and their
applications. In particular, if $K=0$ a spread option is the same as an
option to exchange one asset for another. An explicit solution in this case
has been obtained by Margrabe \cite{Margrabe}. Margrabe's model assumes that 
$S_{1,t}$ and $S_{2,t}$ follow a geometric Brownian motion whose
volatilities $\sigma _{1}$ and $\sigma _{2}$ do not need to be constant, but
the volatility $\sigma $ of $S_{1,t}/S_{2,t}$ is a constant, $\sigma =\left(
\sigma _{1}^{2}+\sigma _{2}^{2}-2\sigma _{1}\sigma _{2}\rho \right) ,$ where 
$\rho $ is the correlation coefficient of the Brownian motions $S_{1,t}$ and 
$S_{2,t}$. Margrabe's formula states that

\begin{equation*}
V=e^{-q_{1}T}S_{1,0}N\left( d_{1}\right) -e^{-q_{2}T}S_{2,0}N\left(
d_{2}\right) ,
\end{equation*}%
where $N$ denotes the cumulative distribution for a standard Normal
distribution, 
\begin{equation*}
d_{1}=\frac{1}{\sigma T^{1/2}}\left( \ln \left( \frac{S_{1,0}}{S_{2,0}}%
\right) +\left( q_{1}-q_{2}+\frac{\sigma }{2}\right) T\right)
\end{equation*}%
and $d_{2}=d_{1}-\sigma T^{1/2}$.

Unfortunately, in the case where $K>0$ and $S_{1,t}$, $S_{2,t}$ are
geometric Brownian motions, no explicit pricing formula is known. In this
case various approximation methods have been developed. There are three main
approaches: Monte Carlo techniques which are most convenient for
high-dimensional situation because the convergence is independent of the
dimension, fast Fourier transform methods studied in \cite{Carr and Madan}
and PDEs. Observe that PDE based methods are suitable if the dimension of
the PDE is low (see, e.g. \cite{Jitse Niesen}, \cite{Duffy}, \cite{Tavella}
and \cite{Wilmott} for more information). The usual PDE's approach is based
on numerical approximation resulting in a large system of ordinary
differential equations which can then be solved numerically.

Approximation formulas usually allow quick calculations. In particular, a
popular among practitioners Kirk formula \cite{Kirk} gives a good
approximation to the spread call (see also Carmona-Durrleman procedure \cite%
{carmona}, \cite{Li Deng Zhou}). Various applications of fast Fourier
transform have been considered in \cite{Dempster and Hong} and \cite{Lord}.

It is well-known that Merton-Black-Scholes theory becomes much more
efficient if additional stochastic factors are introduced. Consequently, it
is important to consider a wider family of L\'{e}vy processes. Stable L\'{e}%
vy processes have been used first in this context by Mandelbrot \cite{man1}
and Fama \cite{f1}.

From the 90th L\'{e}vy processes became very popular (see, e.g., \cite{ms1}, 
\cite{ms2}, \cite{bp1}, \cite{bl1} and references therein). Usually the
reward function has a simple structure, hence the main problem in
computation of integral (\ref{vvv}) is to approximate well the respective
density function. In the present article we develop a general method of
approximation of density functions. This method is saturation free and can
be applied in high-dimensional situation.

\section{\protect\bigskip theoretical background}

\label{sec4}

Let $C(\mathbb{R}^{n})$ be the space of continuous functions on $\mathbb{R}%
^{n}$ and $L_{p}(\mathbb{R}^{n})$ be the usual space of $p$-integrable
functions equipped with the norm 
\begin{equation*}
\Vert f\Vert _{p}=\Vert f\Vert _{L_{p}(\mathbb{R}^{n})}:=\left\{ 
\begin{array}{cc}
\left( \int_{\mathbb{R}^{n}}\left\vert f(\mathbf{x})\right\vert ^{p}d\mathbf{%
x}\right) ^{1/p}, & 1\leq p<\infty , \\ 
\mathrm{ess}\,\,\sup_{\mathbf{x}\in \mathbb{R}^{n}}|f(\mathbf{x})|, & 
p=\infty .%
\end{array}%
\right.
\end{equation*}%
Let $\mathbf{x},\mathbf{y}\in \mathbb{R}^{n}$, $\mathbf{x=}\left(
x_{1},...,x_{n}\right) ,\mathbf{y=}\left( y_{1},...,y_{n}\right) $, and $%
\left\langle \mathbf{x},\mathbf{y}\right\rangle $ be the usual scalar
product in $\mathbb{R}^{n}$, i.e., 
\begin{equation*}
\left\langle \mathbf{x},\mathbf{y}\right\rangle =\sum_{k=1}^{n}x_{k}y_{k}\in 
\mathbb{R}.
\end{equation*}%
For an integrable on $\mathbb{R}^{n}$ function, i.e., $f(\mathbf{x})\in
L_{1}\left( \mathbb{R}^{n}\right) $ define its \textit{Fourier transform} 
\begin{equation*}
\mathbf{F}f(\mathbf{y})=\int_{\mathbb{R}^{n}}\exp \left( -i\left\langle 
\mathbf{x,y}\right\rangle \right) f(\mathbf{x})d\mathbf{x}.
\end{equation*}%
and its formal inverse as%
\begin{equation*}
\left( \mathbf{F}^{-1}f\right) (\mathbf{x})=\frac{1}{\left( 2\pi \right) ^{n}%
}\int_{\mathbb{R}^{n}}\exp \left( i\left\langle \mathbf{x,y}\right\rangle
\right) f(\mathbf{y})d\mathbf{y}.
\end{equation*}%
Remark that in the periodic case the most natural (and in many important
cases optimal in the sense of the respective $n$-widths) method to
approximate sets of smooth functions is to use trigonometric approximations.

In the case of approximation on the whole real line $\mathbb{R}$ the role of
subspaces of trigonometric polynomials play functions from the \textit{%
Wiener spaces} $W_{\sigma }(\mathbb{R})$, i.e., \textit{entire functions}
from $L_{2}(\mathbb{R})$ whose Fourier transform has support $[-\sigma
,\sigma ]$. Such functions have an \textit{exponential type} $\sigma >0$.
Remind that an entire function $f(z)$ defined on the complex plane $\mathbb{C%
}$ can be represented as%
\begin{equation*}
f(z)=\sum_{k=0}^{\infty }c_{k}z^{k}
\end{equation*}%
for any $z\in \mathbb{C}.$ Assume that $f(z)$ has such coefficients $c_{k}$
that 
\begin{equation*}
\overline{\lim }_{k\rightarrow \infty }\left( k!\left\vert c_{k}\right\vert
\right) ^{1/k}=\sigma <\infty .
\end{equation*}%
Then for some constant $M>0$ we have 
\begin{equation*}
\left\vert f\left( z\right) \right\vert \leq \sum_{k=0}^{\infty }\left\vert
c_{k}\right\vert \left\vert z\right\vert ^{k}=\sum_{k=0}^{\infty }\frac{1}{k!%
}\left( \left\vert z\right\vert \left( k!\left\vert c_{k}\right\vert \right)
^{1/k}\right) ^{k}
\end{equation*}%
\begin{equation*}
\leq M\sum_{k=0}^{\infty }\frac{1}{k!}\left( \left\vert \sigma z\right\vert
\right) ^{k}=Me^{\sigma \left\vert z\right\vert }.
\end{equation*}%
We say that a function $f(z)$ defined on the complex plane $\mathbb{C}$ is
of exponential type $\sigma >0$ if there exists a constant $M$ such that for
any $\theta \in \lbrack 0,2\pi )$, 
\begin{equation}
\left\vert f\left( z\right) \right\vert \leq Me^{\sigma
r},\,\,\,z=re^{i\theta }  \label{type}
\end{equation}%
in the limit of $r\rightarrow \infty $. The key role here plays the
classical Paley-Wiener theorem which relates decay properties of a function
at infinity with analyticity of its Fourier transform. It makes use of the
holomorphic Fourier transform defined on the space of square-integrable
functions on $\mathbb{R}$.

\begin{theorem}
\bigskip (Paley-Wiener)\label{t0} Suppose that $F$ is supported in $[-\sigma
,\sigma ]$, so that $F\in L_{2}[-\sigma ,\sigma ]$. Then the holomorphic
Fourier transform 
\begin{equation*}
f(z)=\int_{[-\sigma ,\sigma ]}F(\xi )e^{-iz\xi }d\xi
\end{equation*}%
is an entire function of exponential type $\sigma $ as defined in (\ref{type}%
).
\end{theorem}

Remark that entire functions of exponential type $\sigma $ as an apparatus
of approximation was first considered by Bernstein \cite{bern}.

In the $n$-dimensional settings we use entire functions $f\left( \mathbf{z}%
\right) :\mathbb{C}^{n}\longrightarrow \mathbb{C}$ of $n$ variables $\mathbf{%
z}=\left( z_{1,\cdot \cdot \cdot ,}z_{n}\right) \in \mathbb{C}^{n}$ which
satisfy the condition%
\begin{equation*}
\left\vert f\left( \mathbf{z}\right) \right\vert \leq M\exp \left(
\sum_{k=1}^{n}\sigma _{k}\left\vert z_{k}\right\vert \right) ,\forall 
\mathbf{z\in }\mathbb{C}^{n},
\end{equation*}%
where $M$ is a fixed constant. Here $\mathbf{\sigma :}=\left( \sigma
_{1},\cdot \cdot \cdot ,\sigma _{n}\right) $ is the exponential type of $%
f\left( \mathbf{z}\right) .$ To justify an inversion formula we will need
Planchrel's theorem (see, e.g., \cite{plancherel}).

\begin{theorem}
(Plancherel) The Fourier transform is a linear continuous operator from $%
L_{2}\left( \mathbb{R}^{n}\right) $ onto $L_{2}\left( \mathbb{R}^{n}\right)
. $ The inverse Fourier transform, $\mathbf{F}^{-1},$ can be obtained by
letting 
\begin{equation*}
\left( \mathbf{F}^{-1}g\right) \left( \mathbf{x}\right) =\frac{1}{(2\pi )^{n}%
}\left( \mathbf{F}g\right) \left( -\mathbf{x}\right)
\end{equation*}
for any $g\in L_{2}\left( \mathbb{R}^{n}\right) .$
\end{theorem}

\section{$\mathbf{\protect\lambda }$-deformation of entire basis functions
of exponential type}

\label{lambda}

In this section we discuss a multidimensional generalisation of a well-known
Nyquist-Whitekker-Kotel'nikov-Shannon theorem which explains why Wiener
spaces $W_{\sigma }(\mathbb{R})$ are so important. Observe that
Nyquist-Whitekker-Kotel'nikov-Shannon theorem has its roots in the
Information Theory first developed by Shannon \cite{sha1}, \cite{sha2}, \cite%
{shannon}. There are various extensions of the mention above theorem to a
more general sets of lattice points in $\mathbb{R}^{n}$ \cite{Marks}.
However, these results are aside of our main line of research. In
particular, almost uniformly distributed lattice points in $\mathbb{R}^{n}$
can be used to reproduce trigonometric polynomials just with the spectrum
inside a properly scaled symmetric hyperbolic cross \cite{Kuipers}, \cite%
{Korobov}, \cite{Kushpel-Levesley}. Unfortunately, characteristic exponents
of density functions which correspond to a jump-diffusion process, which can
be used in pricing formulas, should admit an analytic extension into a
proper domain to guarantee the existence of the pricing integral. The shape
of such characteristic exponents is quite far from the hyperbolic cross.
Hence, number theoretic lattice points can not be effective in such
situations. The problem of constructing of lattice points which will
reproduce trigonometric polynomials inside the respective domain which
corresponds to the shape of a given characteristic exponent is a deep
problem of Geometry of Numbers which remains unsolved.

Consider several examples of characteristic functions which illustrate this
situation.

\textbf{Example 1.}\emph{\ Let }$W_{t}^{1}$\emph{\ and }$W_{t}^{2}$\emph{\
are risk-neutral Brownian motions with correlation }$\rho $\emph{\ and }$%
\sigma _{1},\sigma _{2}>0$\emph{. Consider the vector }$S_{t}=\left(
S_{1,t},S_{2,t}\right) $\emph{\ with components}%
\begin{equation*}
S_{k,t}=S_{k,0}\exp \left( \left( r-\sigma _{k}^{2}/2\right) t+\sigma
_{k}W_{t}^{k}\right) ,k=1,2.
\end{equation*}%
\emph{The joint characteristic function of }$X_{T}=\left( \ln S_{1,T},\ln
S_{2,T}\right) $\emph{\ has the form}%
\begin{equation*}
\Phi _{1}\left( \mathbf{u},T\right) =\exp \left( i\left\langle \mathbf{u,}rT%
\mathbf{e-}\frac{T}{2}\mathbf{\sigma }^{2}\right\rangle -\frac{T}{2}%
\left\langle \mathbf{u},\mathbf{\Sigma u}^{T}\right\rangle \right) ,
\end{equation*}%
\emph{where }$u=\left( u_{1},u_{2}\right) ,e=\left( 1,1\right) ,\sigma
^{2}=\left( \sigma _{1}^{2},\sigma _{2}^{2}\right) $\emph{, }%
\begin{equation*}
\mathbf{\Sigma }=\left( 
\begin{array}{cc}
\sigma _{1}^{2} & \sigma _{1}^{2}\sigma _{2}^{2}\rho \\ 
\sigma _{1}^{2}\sigma _{2}^{2}\rho & \sigma _{2}^{2}%
\end{array}%
\right)
\end{equation*}%
\emph{and }$u^{T}$\emph{\ means }$u$\emph{\ transposed. Direct calculation
shows }%
\begin{equation*}
\Phi _{1}\left( \mathbf{u},T\right) =\Phi _{1}\left( u_{1},u_{2}\right)
\end{equation*}%
\begin{equation*}
=\exp \left( -\frac{1}{2}T\left( ir\sigma _{1}^{2}u_{1}+2\rho \sigma
_{1}\sigma _{2}u_{1}u_{2}+\sigma _{1}u_{1}^{2}+ir\sigma _{2}^{2}u_{2}+\sigma
_{2}u_{2}^{2}-2iru_{1}-2iru_{2}\right) \right)
\end{equation*}%
\emph{The parameters are \cite{Dempster and Hong}, \cite{hurd}: }$r=0.1$%
\emph{, }$T=1$\emph{, }$\rho =0.5$\emph{, }$\sigma _{1}=0.2$\emph{, }$\sigma
_{2}=0.1$\emph{. For such set of parameters the function }$\Phi _{1}$\emph{\
simplifies as }%
\begin{equation*}
\Phi _{1}\left( u_{1},u_{2}\right)
\end{equation*}%
\begin{equation*}
=\exp \left( -\frac{1}{2}\left( 0.004iu_{1}+\allowbreak
0.02u_{1}u_{2}+0.2u_{1}^{2}+0.001iu_{2}+0.1\times
u_{2}^{2}-0.2iu_{1}-0.2iu_{2}\right) \right) .
\end{equation*}

\textbf{Example 2.} \emph{Consider a three factor stochastic volatility
model \cite{Dempster and Hong} which is defined as }%
\begin{equation*}
dX_{1}=\left( \left( r-\delta _{1}-\frac{\sigma _{1}^{2}}{2}\right)
dt+\sigma _{1}\upsilon ^{1/2}dW^{1}\right) ,
\end{equation*}%
\begin{equation*}
dX_{2}=\left( \left( r-\delta _{2}-\frac{\sigma _{2}^{2}}{2}\right)
dt+\sigma _{2}\upsilon ^{1/2}dW^{2}\right) ,
\end{equation*}%
\begin{equation*}
d\upsilon =\kappa \left( \mu -\upsilon \right) dt+\sigma _{\upsilon
}\upsilon ^{1/2}dW^{\upsilon },
\end{equation*}%
\emph{where }$dW^{1},dW^{2}$\emph{\ and }$dW^{\upsilon }$\emph{\ have
correlations }%
\begin{equation*}
\mathbb{E}\left[ dW^{1},dW^{2}\right] =\rho dt,
\end{equation*}%
\begin{equation*}
\mathbb{E}\left[ dW^{1},dW^{\upsilon }\right] =\rho _{1}dt,
\end{equation*}%
\begin{equation*}
\mathbb{E}\left[ dW^{2},dW^{\upsilon }\right] =\rho _{2}dt,
\end{equation*}%
$X_{t}=\left( \log S_{1,t},\log S_{2,t}\right) $\emph{\ and }$\upsilon _{t}$%
\emph{\ is the squared volatility. The characteristic function has the form }%
\begin{equation*}
\Phi _{2}\left( \mathbf{u}\right) =\Phi _{2}\left( u_{1},u_{2}\right)
=\left( ix\ln S_{1,0}+iy\ln S_{2,0}+\left( \frac{2\omega \left( 1-e^{-\theta
T}\right) }{2\theta -\left( \theta -\gamma \right) \left( 1-e^{-\theta
T}\right) }\right) \upsilon _{0}\right.
\end{equation*}%
\begin{equation*}
+i\left\langle \mathbf{u,}\left( r\mathbf{e}-\mathbf{\delta }\right)
\right\rangle T
\end{equation*}%
\begin{equation*}
\left. -\frac{\kappa \mu }{\sigma _{\upsilon }^{2}}\left( 2\log \left( \frac{%
2\theta -\left( \theta -\gamma \right) \left( 1-e^{-\theta T}\right) }{%
2\theta }\right) +\left( \theta -\gamma \right) T\right) \right) ,
\end{equation*}%
\emph{where}%
\begin{equation*}
\omega :=-\frac{1}{2}\left( \left( \sigma _{1}^{2}u_{1}^{2}+\sigma
_{2}^{2}u_{2}^{2}+2\rho \sigma _{1}\sigma _{2}u_{1}u_{2}\right) +i\left(
\sigma _{1}^{2}u_{1}+\sigma _{2}^{2}u_{2}\right) \right) ,
\end{equation*}%
\begin{equation*}
\gamma :=\kappa -i\left( \rho _{1}\sigma _{1}u_{1}+\rho _{2}\sigma
_{2}u_{2}\right) \sigma _{\upsilon }
\end{equation*}%
\begin{equation*}
\theta :=\left( \gamma ^{2}-2\sigma _{\upsilon }^{2}\omega \right) .
\end{equation*}%
\emph{Let us fix parameters as in \cite{Dempster and Hong}, p.16: }$%
r=0.1,T=1,\rho =0.5,\rho =0.5,\rho _{1}=0.25,\rho _{2}=-0.5,\delta
_{1}=0.05,\delta _{2}=0.05,\sigma _{1}=0.5,\sigma _{2}=1.0,\upsilon
_{0}=0.04,\kappa =1,\mu =0.04,\sigma _{\upsilon
}=0.05,S_{1,0}=96,S_{2,0}=100.$

\textbf{Example 3.} \emph{Following \cite{m3} consider the }$VG$\emph{\
process. The L\'{e}vy measure in this case is }%
\begin{equation*}
\Pi \left( x\right) =\frac{\lambda \left( e^{-a_{+}x}\chi _{\left[ 0,\infty
\right) }\left( x\right) +e^{a_{-}x}\chi _{\left( -\infty ,0\right] }\left(
x\right) \right) }{x},\lambda >0,a_{+}>0,a_{-}>0,
\end{equation*}%
\emph{where }%
\begin{equation*}
\chi _{A}\left( x\right) :=\left\{ 
\begin{array}{cc}
1, & x\in A, \\ 
0, & x\notin A,%
\end{array}%
\right.
\end{equation*}%
$A\in R$\emph{\ and the characteristic function is}%
\begin{equation*}
\Phi _{Y_{t}}\left( u\right) =\left( 1+i\left( \frac{1}{a_{-}}-\frac{1}{a_{+}%
}\right) u+\frac{u^{2}}{a_{-}a_{+}}\right) ^{-\lambda t}.
\end{equation*}%
\emph{Let }$Y_{k,t}$\emph{, }$k=1,2,3$\emph{\ be three independent }$VG$%
\emph{\ processes with common parameters }$a_{+}$\emph{, }$a_{-}$\emph{, }$%
\lambda _{1}=\lambda _{2}=\left( 1-\alpha \right) \lambda $\emph{, }$\lambda
_{3}=\alpha \lambda $\emph{, }$\alpha \in \left[ 0,1\right] .$\emph{\ The }$%
\log $\emph{\ return }$X_{k,t}=\log S_{k,t},k=1,2$\emph{\ is given by }%
\begin{equation*}
X_{k,t}=X_{k,0}+Y_{k,t}+Y_{3,t},k=1,2.
\end{equation*}%
\emph{The characteristic function has the form }%
\begin{equation*}
\Phi _{3}\left( \mathbf{u},T\right) =\Phi _{3}\left( u_{1},u_{2},T\right)
\end{equation*}%
\begin{equation*}
=\left( 1+i\left( \frac{1}{a_{-}}-\frac{1}{a_{+}}\right) \left(
u_{1}+u_{2}\right) +\frac{\left( u_{1}+u_{2}\right) ^{2}}{a_{-}a_{+}}\right)
^{-\alpha \lambda T}
\end{equation*}%
\begin{equation*}
\times \left( 1+i\left( \frac{1}{a_{-}}-\frac{1}{a_{+}}\right) u_{1}+\frac{%
u_{1}^{2}}{a_{-}a_{+}}\right) ^{-\left( 1-\alpha \right) \lambda T}
\end{equation*}%
\begin{equation*}
\times \left( 1+i\left( \frac{1}{a_{-}}-\frac{1}{a_{+}}\right) u_{2}+\frac{%
u_{2}^{2}}{a_{-}a_{+}}\right) ^{-\left( 1-\alpha \right) \lambda T}.
\end{equation*}%
\emph{Let us put }$T=1,a_{+}=2,a_{-}=3,\lambda =1,\alpha =0.5$\emph{.}

Observe that all three examples show high concentration of characteristic
functions around the origin. 

Consider a general case now. Let $\mathbf{a}$ be a fixed positive vector in $%
\mathbb{R}^{n}$, i.e. $\mathbf{a}=(a_{1},\cdots ,a_{n})\in \mathbb{R}^{n}$, $%
a_{k}>0$, $1\leq k\leq n$ and $\mathrm{A}=\mathrm{diag}(a_{1}^{-1},\cdots
,a_{n}^{-1})$ be a diagonal matrix generated by $\mathbf{a}$. Consider the
set of points in $\mathbb{R}^{n}$. 
\begin{equation*}
\Omega _{\mathbf{a}}=\left\{ \mathbf{z}_{\mathbf{m}}=\mathrm{A}\mathbf{m}^{%
\mathrm{T}}|\,\,\mathbf{m}\in \mathbb{Z}^{n}\right\} ,\,\,\,\mathbf{m}\in 
\mathbb{Z}^{n},
\end{equation*}%
where $\mathbf{m}^{\mathrm{T}}$ means transpose of $\mathbf{m}$. Observe
that 
\begin{equation*}
\mathbf{z}_{\mathbf{m}}=\left( \frac{m_{1}}{a_{1}},\cdots ,\frac{m_{n}}{a_{n}%
}\right) \in \mathbb{R}^{n}.
\end{equation*}%
for any fixed $\mathbf{m}\in \mathbb{Z}^{n}.$ Let 
\begin{equation*}
Q_{\mathbf{a}}:=\left\{ \mathbf{x}\left\vert {}\right. \,\mathbf{x}%
=(x_{1},\cdots ,x_{n})\in \mathbb{R}^{n},\,\,|x_{k}|\leq a_{k},\,1\leq k\leq
n\right\} .
\end{equation*}%
Denote by $W_{\mathbf{a}}(\mathbb{R}^{n})$ the space of functions $f\in
L_{2}(\mathbb{R}^{n})$ such that $\mathrm{supp}\,\mathbf{F}f\subset Q_{%
\mathbf{a}}.$ \ 

Let $C\left( \mathbb{C}^{n}\right) $ and $C(\mathbb{R}^{n})$ be the spaces
of continuous functions on $\mathbb{C}^{n}$ and $\mathbb{R}^{n}$
respectively. We construct a family of linear operators $\mathbf{P}_{\mathbf{%
\lambda }_{\mathbf{a}}}$,%
\begin{equation*}
\begin{array}{ccc}
\mathbf{P}_{\mathbf{\lambda }_{\mathbf{a}}}:C\left( \mathbb{C}^{n}\right) & 
\longrightarrow & W_{2\mathbf{a}}(\mathbb{R}^{n})+iW_{2\mathbf{a}}(\mathbb{R}%
^{n})\subset C(\mathbb{R}^{n})+iC(\mathbb{R}^{n}) \\ 
f\left( \mathbf{z}\right) & \longmapsto & \left( \mathbf{P}_{\mathbf{\lambda 
}_{\mathbf{a}}}f\right) \left( \mathbf{z}\right)%
\end{array}%
\end{equation*}%
such that 
\begin{equation*}
\left\Vert \mathbf{P}_{\mathbf{\lambda }_{\mathbf{a}}}\mathbf{|}C\left( 
\mathbb{C}^{n}\right) \longrightarrow C(\mathbb{R}^{n})+iC(\mathbb{R}%
^{n})\right\Vert <\infty
\end{equation*}%
and $\left( \mathbf{P}_{\mathbf{\lambda }_{\mathbf{a}}}f\right) \left( 
\mathbf{z}\right) =f\left( \mathbf{z}\right) $ for any $f\left( \mathbf{z}%
\right) \in W_{\mathbf{a}}(\mathbb{R}^{n}).$ The sign "$+$" means the
Minkowski sum of two vector spaces $C(\mathbb{R}^{n})$ and $iC(\mathbb{R}%
^{n})$ endowed with the induced topology of $C\left( \mathbb{C}^{n}\right)
\supset C(\mathbb{R}^{n})+iC(\mathbb{R}^{n}).$

\begin{theorem}
\label{t1} Let $f\left( \mathbf{z}\right) \in W_{\mathbf{a}}(\mathbb{R}^{n})$
and $\mathbf{\lambda }_{\mathbf{a}}:\mathbb{R}^{n}\longrightarrow \mathbb{R}$
be any continuous function such that $\mathbf{\lambda }_{\mathbf{a}}\left( 
\mathbf{y}\right) =1$ if $\mathbf{y}\in Q_{\mathbf{a}}$ and $\mathbf{\lambda 
}_{\mathbf{a}}\left( \mathbf{y}\right) =0$ if $\mathbf{y}\in $ $\mathbb{R}%
^{n}\setminus Q_{2\mathbf{a}},$ then 
\begin{equation*}
f\left( \mathbf{x}\right) =\sum_{\mathbf{m}\in \mathbb{Z}^{n}}f\left( \pi 
\mathrm{A}\mathbf{m}^{\mathrm{T}}\right) J_{\mathbf{m},\mathbf{\lambda }_{%
\mathbf{a}}}\left( \mathbf{x}\right)
\end{equation*}%
\begin{equation*}
=\sum_{m_{1}\in \mathbb{Z}}\cdot \cdot \cdot \sum_{m_{n}\in \mathbb{Z}%
}f\left( \pi \frac{m_{1}}{a_{1}},\cdot \cdot \cdot ,\pi \frac{m_{n}}{a_{n}}%
\right) J_{m_{1},\cdot \cdot \cdot ,m_{n},\mathbf{\lambda }_{a_{1},\cdot
\cdot \cdot ,a_{n}}}\left( x_{1},\cdot \cdot \cdot ,x_{n}\right),
\end{equation*}%
where%
\begin{equation*}
J_{\mathbf{m},\mathbf{\lambda }_{\mathbf{a}}}\left( \mathbf{x}\right)
=\pi^{n} \mathrm{det}\mathrm{A} \left( \mathbf{F}^{-1}\mathbf{\lambda }_{%
\mathbf{a}}\right) \left( \mathbf{x}-i\pi \frac{\mathbf{m}}{\mathbf{a}}%
\right)
\end{equation*}
\begin{equation*}
=2^{-n} \mathrm{det}\mathrm{A} \,\left( \mathbf{F\lambda }_{\mathbf{a}}
\right) \left( -\mathbf{x}+\pi \frac{\mathbf{m}}{\mathbf{a}}\right)
\end{equation*}
\end{theorem}

\begin{proof}
For any $f\in $ $W_{\mathbf{a}}(\mathbb{R}^{n})$ we have 
\begin{equation*}
f(\mathbf{x})=\frac{1}{(2\pi )^{n}}\,\int_{\mathbb{R}^{n}}\,\left( \mathbf{F}%
f\right) (\mathbf{y})e^{i\mathbf{y}\cdot \mathbf{x}}d\mathbf{y}
\end{equation*}%
\begin{equation*}
=\frac{1}{(2\pi )^{n}}\,\int_{Q_{2\mathbf{a}}}\,\mathbf{\lambda }_{\mathbf{a}%
}\left( \mathbf{y}\right) \left( \mathbf{F}f\right) (\mathbf{y})e^{i\mathbf{y%
}\cdot \mathbf{x}}d\mathbf{y}
\end{equation*}%
because $\text{supp}\,\mathbf{F}f\subset Q_{\mathbf{a}}$ and $\mathbf{%
\lambda }_{\mathbf{a}}\left( \mathbf{y}\right) =1$ if $\mathbf{y\in }Q_{%
\mathbf{a}}$ and $\mathbf{\lambda }_{\mathbf{a}}\left( \mathbf{y}\right) =0$
if $\mathbf{y\in }\mathbb{R}^{n}\setminus Q_{2\mathbf{a}}$. Since the set 
\begin{equation*}
\varrho _{\mathbf{m}}(\mathbf{y,a}):=2^{-n/2}\left( \det \mathrm{A}\right)
^{1/2}\exp \left( i\pi \left\langle \mathrm{A}\mathbf{m}^{\mathrm{T}},%
\mathbf{y}\right\rangle \right)
\end{equation*}%
\begin{equation*}
=\left( \prod_{k=1}^{n}\,\left( \frac{1}{(2a_{k})^{1/2}}\right) \right)
\,\prod_{k=1}^{n}\,\exp \left( \frac{i\pi }{a_{k}}\,\,m_{k}y_{k}\right)
,\,\,\,\mathbf{m}\in \mathbb{Z}^{n}
\end{equation*}%
is an orthonormal basis in $L_{2}\left( Q_{\mathbf{a}}\right) $ then $\left( 
\mathbf{F}f\right) (\mathbf{y})$ can be represented as 
\begin{equation*}
\left( \mathbf{F}f\right) (\mathbf{y})=\sum_{\mathbf{m}\in \mathbb{Z}%
^{n}}\,\alpha _{\mathbf{m}}\varrho _{\mathbf{m}}(\mathbf{y,a}).
\end{equation*}%
Remind that $f\in $ $W_{\mathbf{a}}(\mathbb{R}^{n})\subset L_{2}(\mathbb{R}%
^{n})$. We understand the convergence in $L_{2}\left( Q_{\mathbf{a}}\right) $
in the sense that 
\begin{equation*}
\lim_{N\rightarrow \infty }\left\Vert \left( \mathbf{F}f\right) (\mathbf{y}%
)-\sum_{\mathbf{m}\in NQ_{\mathbf{a}}}\,\alpha _{\mathbf{m}}\varrho _{%
\mathbf{m}}(\mathbf{y,a})\right\Vert _{L_{2}\left( Q_{\mathbf{a}}\right) }=0.
\end{equation*}%
Observe that instead of $Q_{\mathbf{a}}$ we can take any neighborhood of $%
\mathbf{0}\in \mathbb{R}^{n}$. Using Plancherel's theorem we find that 
\begin{equation*}
\alpha _{\mathbf{m}}=\,\int_{Q_{\mathbf{a}}}\left( \mathbf{F}f\right) (%
\mathbf{y})\varrho _{\mathbf{m}}(-\mathbf{y,a})d\mathbf{y}
\end{equation*}%
\begin{equation*}
=\,\int_{Q_{\mathbf{a}}}\left( \mathbf{F}f\right) (\mathbf{y})\overline{%
\varrho _{\mathbf{m}}(\mathbf{y,a})}d\mathbf{y}
\end{equation*}%
\begin{equation*}
=\,\int_{\mathbb{R}^{n}}\left( \mathbf{F}f\right) (\mathbf{y})\varrho _{%
\mathbf{m}}(-\mathbf{y,a})d\mathbf{y}
\end{equation*}%
\begin{equation*}
=\,(2\pi )^{n}2^{-n/2}\left( \det \mathrm{A}\right) ^{1/2}\left( \mathbf{F}%
^{-1}\circ \mathbf{F}f\right) \left( -\pi \mathrm{A}\mathbf{m}^{\mathrm{T}%
}\right)
\end{equation*}%
\begin{equation*}
=(2\pi )^{n}2^{-n/2}\left( \det \mathrm{A}\right) ^{-1/2}f\left( -\pi 
\mathrm{A}\mathbf{m}^{\mathrm{T}}\right) .
\end{equation*}%
Applying Plancherel's theorem again we get 
\begin{equation*}
f(\mathbf{x})=\frac{1}{(2\pi )^{n}}\,\int_{\mathbb{R}^{n}}\,\sum_{\mathbf{m}%
\in \mathbb{Z}^{n}}\,(2\pi )^{n}2^{-n/2}\left( \det \mathrm{A}\right)
^{1/2}f\left( -\pi \mathrm{A}\mathbf{m}^{\mathrm{T}}\right)
\end{equation*}%
\begin{equation*}
\times \mathbf{\lambda }_{\mathbf{a}}\left( \mathbf{y}\right) \,\varrho _{%
\mathbf{m}}(\mathbf{y,a})\,e^{i\left\langle \mathbf{x,y}\right\rangle }\,d%
\mathbf{y}
\end{equation*}%
\begin{equation*}
=\,\int_{Q_{2\mathbf{a}}}\,\sum_{\mathbf{m}\in \mathbb{Z}^{n}}2^{-n/2}\left(
\det \mathrm{A}\right) ^{1/2}\,f\left( -\pi \mathrm{A}\mathbf{m}^{\mathrm{T}%
}\right)
\end{equation*}%
\begin{equation*}
\times 2^{-n/2}\left( \det \mathrm{A}\right) ^{1/2}\,\mathbf{\lambda }_{%
\mathbf{a}}\left( \mathbf{y}\right) \exp \left( i\left\langle \mathbf{x},%
\mathbf{y}\right\rangle +i\pi \left\langle \mathrm{A}\mathbf{m}^{\mathrm{T}},%
\mathbf{y}\right\rangle \right) \,d\mathbf{y}
\end{equation*}%
\begin{equation*}
=2^{-n}\det \mathrm{A}\,\sum_{\mathbf{m}\in \mathbb{Z}^{n}}\,f\left( -\pi 
\mathrm{A}\mathbf{m}^{\mathrm{T}}\right) \,\int_{Q_{2\mathbf{a}}}\mathbf{%
\lambda }_{\mathbf{a}}\left( \mathbf{y}\right) \exp \left( i\left\langle 
\mathbf{x},\mathbf{y}\right\rangle +i\pi \left\langle \mathrm{A}\mathbf{m}^{%
\mathrm{T}},\mathbf{y}\right\rangle \right) \,d\mathbf{y}
\end{equation*}%
Changing the index of summation and simplifying we get 
\begin{equation*}
f(\mathbf{x})=\sum_{\mathbf{m}\in \mathbb{Z}^{n}}\,f\left( \pi \mathrm{A}%
\mathbf{m}^{\mathrm{T}}\right) J_{\mathbf{m,\lambda }_{\mathbf{a}}}(\mathbf{x%
}),
\end{equation*}%
where 
\begin{equation*}
J_{\mathbf{m,\lambda }_{\mathbf{a}}}(\mathbf{x})=2^{-n}\det \mathrm{A}%
\int_{Q_{2\mathbf{a}}}\mathbf{\lambda }_{\mathbf{a}}\left( \mathbf{y}\right)
\exp \left( i\left\langle \mathbf{x},\mathbf{y}\right\rangle -i\pi
\left\langle \mathrm{A}\mathbf{m}^{\mathrm{T}},\mathbf{y}\right\rangle
\right) \,d\mathbf{y}
\end{equation*}%
\begin{equation*}
=2^{-n}\det \mathrm{A}\,\int_{Q_{2\mathbf{a}}}\mathbf{\lambda }_{\mathbf{a}%
}\left( \mathbf{y}\right) \exp \left( i\left\langle \mathbf{x}-\pi \mathrm{A}%
\mathbf{m}^{\mathrm{T}},\mathbf{y}\right\rangle \right) \,d\mathbf{y}
\end{equation*}%
\begin{equation*}
=\,2^{-n}\det \mathrm{A}\left( 2\pi \right) ^{n}\,\left( \mathbf{F}^{-1}%
\mathbf{\lambda }_{\mathbf{a}}\right) \left( \mathbf{x}-\pi \mathrm{A}%
\mathbf{m}^{\mathrm{T}}\right) .
\end{equation*}%
\begin{equation*}
=2^{-n}\det \mathrm{A}\,\left( \mathbf{F\lambda }_{\mathbf{a}}\right) \left(
-\mathbf{x}+\pi \mathrm{A}\mathbf{m}^{\mathrm{T}}\right) .
\end{equation*}
\end{proof}

Consider a particular form of $\mathbf{\lambda }$-deformation. Let 
\begin{equation}
\mathbf{\lambda }_{\mathbf{a}}\left( \mathbf{x}\right) =\mathbf{\lambda }%
_{a_{1},\cdot \cdot \cdot ,a_{n}}\left( x_{1},\cdot \cdot \cdot
,x_{n}\right) :=\dprod\limits_{k=1}^{n}\lambda _{a_{k}}\left( x_{k}\right) ,
\label{lambda11}
\end{equation}%
where 
\begin{equation*}
\lambda _{a_{k}}\left( x_{k}\right) :=\left\{ 
\begin{array}{cc}
0, & x\leq -a, \\ 
2a^{-1}x+2, & -a\leq x\leq -a/2, \\ 
1, & -a/2\leq x<a/2, \\ 
-2a^{-1}x+2 & a/2\leq x<a, \\ 
0, & x\geq a,%
\end{array}%
\right. ,
\end{equation*}%
Direct calculation shows that

\begin{equation*}
\left( \mathbf{F}\lambda _{a_{k}}\right) \left( y_{k}\right) =\left(
\int_{-a_{k}}^{-a_{k}/2}+\int_{-a_{k}/2}^{a_{k}/2}+\int_{a_{k}/2}^{a_{k}}%
\right) e^{-ixy_{k}}\lambda _{a_{k}}\left( x\right) dx
\end{equation*}%
\begin{equation*}
=\int_{-a_{k}}^{-a_{k}/2}e^{-ixy}\left( 2x+2\right)
dx+\int_{-a_{k}/2}^{a_{k}/2}e^{-ixy}dx+\int_{a_{k}/2}^{a_{k}}e^{-ixy}\left(
-2x+2\right) dx
\end{equation*}%
\begin{equation*}
=\frac{2}{a_{k}y^{2}}e^{-i\left( -\frac{a_{k}}{2}\right) y}\left( i\left( -%
\frac{a_{k}}{2}\right) y+1\right) -\frac{2}{a_{k}y^{2}}e^{-i\left(
-a_{k}\right) y}\left( i\left( -a_{k}\right) y+1\right)
\end{equation*}%
\begin{equation*}
+2\left( \frac{i}{y}e^{-i\left( -\frac{a_{k}}{2}\right) y}-\frac{i}{y}%
e^{-i\left( -a_{k}\right) y}\right)
\end{equation*}%
\begin{equation*}
+\frac{i}{y}e^{-i\left( \frac{a_{k}}{2}\right) y}-\allowbreak \frac{i}{y}%
e^{-i\left( -\frac{a_{k}}{2}\right) y}
\end{equation*}%
\begin{equation*}
+\frac{-2}{a_{k}y^{2}}e^{-ia_{k}y}\left( ia_{k}y+1\right) -\frac{-2}{%
a_{k}y^{2}}e^{-i\left( \frac{a_{k}}{2}\right) y}\left( i\left( \frac{a_{k}}{2%
}\right) y+1\right)
\end{equation*}%
\begin{equation*}
+2\left( \frac{i}{y}e^{-ia_{k}y}-\frac{i}{y}e^{-i\left( a_{k}/2\right)
y}\right)
\end{equation*}%
\begin{equation*}
=\frac{i}{y}e^{\frac{1}{2}ia_{k}y}-\frac{2}{a_{k}y^{2}}e^{ia_{k}y}+\frac{2}{%
a_{k}y^{2}}e^{\frac{1}{2}ia_{k}y}
\end{equation*}%
\begin{equation*}
+\frac{i}{y}e^{-i\left( \frac{a_{k}}{2}\right) y}-\allowbreak \frac{i}{y}%
e^{-i\left( -\frac{a_{k}}{2}\right) y}
\end{equation*}%
\begin{equation*}
=\frac{2}{a_{k}y^{2}}e^{-\frac{1}{2}ia_{k}y}-\frac{2}{a_{k}y^{2}}%
e^{-ia_{k}y}-\frac{i}{y}e^{-\frac{1}{2}ia_{k}y}
\end{equation*}%
\begin{equation*}
=\frac{4}{a_{k}y^{2}}\left( \cos \frac{1}{2}a_{k}y-\cos a_{k}y\right) .
\end{equation*}

\begin{lemma}
\bigskip 
\begin{equation*}
\left\Vert \mathbf{P}_{\lambda _{a}}\left\vert C\left( \mathbb{R}\right)
\rightarrow C\left( \mathbb{R}\right) \right. \right\Vert <\allowbreak
2.\,\allowbreak 834.
\end{equation*}
\end{lemma}

\begin{proof}
\begin{equation*}
\left\Vert \mathbf{P}_{\lambda _{a}}\left\vert C\left( \mathbb{R}\right)
\rightarrow C\left( \mathbb{R}\right) \right. \right\Vert :=\sup \left\{ 
\mathbf{P}_{\lambda _{a}}f\left\vert \left\Vert f\right\Vert _{C\left( 
\mathbb{R}\right) }\leq 1\right. \right\}
\end{equation*}%
\begin{equation*}
=\sup \left\{ \sup_{x\in \mathbb{R}}\sum_{m\in \mathbb{Z}}f\left( \frac{\pi m%
}{a}\right) J_{m,\lambda _{a}}\left( x\right) \left\vert \left\Vert
f\right\Vert _{C\left( \mathbb{R}\right) }\leq 1\right. \right\}
\end{equation*}%
\begin{equation*}
\leq \sup_{x\in \mathbb{R}}\sum_{m\in \mathbb{Z}}\left\vert J_{m,\lambda
_{a}}\left( x\right) \right\vert
\end{equation*}%
\begin{equation*}
=\sup_{x\in \mathbb{R}}\sum_{m\in \mathbb{Z}}\left\vert 2^{-1}\left(
a^{-1}\right) ^{-1}\left( \mathbf{F}\lambda _{a}\right) \left( -x+\frac{\pi m%
}{a}\right) \right\vert
\end{equation*}%
\begin{equation*}
=2^{-1}a\sup_{x\in \mathbb{R}}\sum_{m\in \mathbb{Z}}\left\vert \frac{4}{%
a\left( ax-\pi m\right) ^{2}}\left( \cos \left( \frac{ax-\pi m}{2}\right)
-\cos \left( ax-\pi m\right) \right) \right\vert
\end{equation*}%
\begin{equation*}
=2\sup_{x\in \mathbb{R}}\sum_{m\in \mathbb{Z}}\left\vert \frac{1}{\left(
ax-\pi m\right) ^{2}}\left( \cos \left( \frac{ax-\pi m}{2}\right) -\cos
\left( ax-\pi m\right) \right) \right\vert .
\end{equation*}%
Observe that the function 
\begin{equation*}
\varphi \left( x\right) :=2\sum_{m\in \mathbb{Z}}\left\vert \frac{1}{\left(
ax-\pi m\right) ^{2}}\left( \cos \left( \frac{ax-\pi m}{2}\right) -\cos
\left( ax-\pi m\right) \right) \right\vert
\end{equation*}%
is $\pi /a$-periodic. Consequently 
\begin{equation*}
\sup \left\{ \varphi \left( x\right) \left\vert x\in \mathbb{R}\right.
\right\} =\sup \left\{ \varphi \left( x\right) \left\vert x\in \left[ 0,\pi
/a\right) \right. \right\}
\end{equation*}%
\begin{equation*}
\leq 2\left( \sup_{x\in \left[ 0,\pi /a\right) }\sum_{\left\vert
m\right\vert \geq 2}+\sup_{x\in \left[ 0,\pi /a\right) }\sum_{m=\left\{
-1,0,1\right\} }\right)
\end{equation*}%
\begin{equation*}
\left\vert \frac{1}{\left( ax-\pi m\right) ^{2}}\left( \cos \left( \frac{%
ax-\pi m}{2}\right) -\cos \left( ax-\pi m\right) \right) \right\vert .
\end{equation*}%
Clearly,%
\begin{equation*}
2\sup_{x\in \left[ 0,\pi /a\right) }\sum_{m\leq -2}\left\vert \frac{1}{%
\left( ax-\pi m\right) ^{2}}\left( \cos \left( \frac{ax-\pi m}{2}\right)
-\cos \left( ax-\pi m\right) \right) \right\vert
\end{equation*}%
\begin{equation*}
\leq 4\sum_{m\geq 2}\frac{1}{\pi ^{2}m^{2}}=\frac{4}{\pi ^{2}}\left( \frac{%
\pi ^{2}}{6}-1\right) .
\end{equation*}

Similarly%
\begin{equation*}
2\sup_{x\in \left[ 0,\pi /a\right) }\sum_{m\geq 2}\left\vert \frac{1}{\left(
ax-\pi m\right) ^{2}}\left( \cos \left( \frac{ax-\pi m}{2}\right) -\cos
\left( ax-\pi m\right) \right) \right\vert
\end{equation*}%
\begin{equation*}
\leq 4\sum_{m\geq 1}\frac{1}{\pi ^{2}m^{2}}=\frac{4}{\pi ^{2}}\frac{\pi ^{2}%
}{6}=\frac{2}{3}.
\end{equation*}%
and

\begin{equation*}
2\sup_{x\in \left[ 0,\pi /a\right) }\sum_{m=\left\{ -1,0,1\right\}
}\left\vert \frac{1}{\left( ax-\pi m\right) ^{2}}\left( \cos \left( \frac{%
ax-\pi m}{2}\right) -\cos \left( ax-\pi m\right) \right) \right\vert
\end{equation*}%
\begin{equation*}
\leq 2\left( 2\cdot 0.375+2\cdot \frac{1}{\pi ^{2}}\right) \approx
\allowbreak 1.\,\allowbreak 905\,284\,735.
\end{equation*}

Hence 
\begin{equation*}
\left\Vert \mathbf{P}_{\lambda _{a}}\left\vert C\left( \mathbb{R}\right)
\rightarrow C\left( \mathbb{R}\right) \right. \right\Vert \leq \frac{4}{\pi
^{2}}\left( \frac{\pi ^{2}}{6}-1\right) +\frac{2}{3}+1.\,\allowbreak
905\,284\,735\approx \allowbreak 2.\,\allowbreak 834.
\end{equation*}
\end{proof}

\begin{corollary}
\label{norm 1} 
\begin{equation*}
\left\Vert \mathbf{P}_{\lambda _{\mathbf{a}}}\left\vert C\left( \mathbb{R}%
^{n}\right) \rightarrow C\left( \mathbb{R}^{n}\right) \right. \right\Vert
\leq \dprod\limits_{k=1}^{n}\left\Vert \mathbf{P}_{\lambda
_{a_{k}}}\left\vert C\left( \mathbb{R}\right) \rightarrow C\left( \mathbb{R}%
\right) \right. \right\Vert \leq 2.\,\allowbreak 834^{n}
\end{equation*}%
and 
\begin{equation*}
\left\Vert \mathbf{P}_{\lambda _{\mathbf{a}}}\left\vert C\left( \mathbb{C}%
^{n}\right) \rightarrow C(\mathbb{C}^{n}))\right. \right\Vert \leq 2\times
2.\,\allowbreak 834^{n}.
\end{equation*}
\end{corollary}

Another example of $\lambda $-deformation is given by the function 
\begin{equation*}
\phi \left( x\right) =\frac{1}{\omega }\int_{-1}^{x}\left( g\left( 4t+\frac{3%
}{4}\right) -g\left( 4t-\frac{3}{4}\right) \right) dt,
\end{equation*}%
where%
\begin{equation*}
g\left( t\right) =\left\{ 
\begin{array}{cc}
\exp \left( -\left( 1-t^{2}\right) ^{-1}\right) , & t\in \left[ -1,1\right] ,
\\ 
0, & t\in \mathbb{R}\setminus \left[ -1,1\right] ,%
\end{array}%
\right.
\end{equation*}%
and%
\begin{equation*}
\omega =\frac{1}{4}\int_{-1}^{1}g\left( t\right) dt.
\end{equation*}%
Clearly, Fourier transform of $\phi \left( x\right) $ is an entire function
of type $1$ since $\phi \left( x\right) \equiv 0,x\in \mathbb{R}\setminus %
\left[ -1,1\right] .$ Applying the method of saddle point it is possible to
show that 
\begin{equation*}
\mathbf{F}g\left( y\right) \approx 2\func{Re}\left( \left( \frac{-i\pi }{%
\left( 2i\right) ^{1/2}}\right) ^{1/2}y^{-3/2}\exp \left( iy-4^{-1}-\left(
2iy\right) ^{1/2}\right) \right) \asymp y^{-3/2}\exp \left( -y^{1/2}\right) .
\end{equation*}%
Hence%
\begin{equation*}
\mathbf{F}\phi \left( y\right) \asymp \left\vert y\right\vert ^{-5/2}\exp
\left( -2\left\vert y\right\vert ^{1/2}\right) .
\end{equation*}%
Let 
\begin{equation*}
\lambda _{\mathbf{a}}\left( \mathbf{x}\right) =\dprod\limits_{k=1}^{n}\phi
\left( \frac{y_{k}}{a_{k}}\right) ,\mathbf{a}=\left( a_{1},\cdot \cdot \cdot
,a_{n}\right)
\end{equation*}%
then%
\begin{equation*}
J_{\mathbf{m,}\lambda _{\mathbf{a}}}\left( \mathbf{x}\right) =\left(
\dprod\limits_{k=1}^{n}\frac{1}{2a_{k}}\right) \mathbf{F}^{-1}\lambda _{%
\mathbf{a}}\left( \mathbf{x}-\pi \mathrm{A}\mathbf{m}^{\mathrm{T}}\right)
\end{equation*}%
and%
\begin{equation*}
=2^{-n}\mathbf{F}^{-1}\lambda _{\mathbf{1}}\left( \mathrm{A}^{-1}\mathbf{x}%
-\pi \mathbf{m}^{\mathrm{T}}\right) =2^{-n}\mathbf{F}^{-1}\lambda _{\mathbf{1%
}}\left( \mathrm{A}^{-1}\left( \mathbf{x}-\pi \mathrm{A}\mathbf{m}^{\mathrm{T%
}}\right) \right) ,\mathbf{1=}\left( 1,\cdot \cdot \cdot ,1\right) .
\end{equation*}%
Therefore%
\begin{equation*}
J_{\mathbf{m,}\lambda _{\mathbf{a}}}\left( \mathbf{x}\right) \asymp
\left\vert a_{k}x_{k}\right\vert ^{-5/2}\exp \left( -2a_{k}\left\vert x_{k}-%
\frac{\pi m_{k}}{a_{k}}\right\vert ^{1/2}\right) ,x_{k}\rightarrow \infty .
\end{equation*}%
and the norm of the interpolation operator can be bounded as 
\begin{equation*}
\left\Vert \mathbf{P}_{\lambda _{\mathbf{a}}}\left\vert C\left( \mathbb{C}%
^{n}\right) \rightarrow C(\mathbb{C}^{n}))\right. \right\Vert \leq 2\sup_{%
\mathbf{x}\in \mathbb{R}}\sum_{\mathbf{m}\in \mathbb{Z}^{n}}\left\vert J_{%
\mathbf{m},\lambda _{\mathbf{a}}}\left( \mathbf{x}\right) \right\vert \ll 1.
\end{equation*}

\section{approximation of density functions}

\label{approximation}

Consider the set $A_{\infty ,\delta }UM$ of analytic functions $f\left(
z\right) =u\left( x,y\right) +iv\left( x,y\right) ,z=x+iv$ in the strip $%
\left\vert \func{Im}z\right\vert \leq \delta $ such that $\left\vert u\left(
x,y\right) \right\vert \leq M$ for any $z\in \left\{ z|\left\vert \func{Im}%
z\right\vert \leq \delta \right\} $ and $f\left( x\right) \in \mathbb{R}$
for any $x\in \mathbb{R}$. It is known (see, \cite{timan}, p. 150) that for
almost all $x$ there are limits 
\begin{equation*}
\lim_{y\rightarrow \delta }u\left( x,y\right) =\lim_{y\rightarrow -\delta
}u\left( x,y\right) =g\left( x\right)
\end{equation*}
and such functions can be represented in the form%
\begin{equation*}
f\left( z\right) =\frac{1}{2\delta }\int_{\mathbb{R}}\left( \cosh \left( 
\frac{\pi \left( z-x\right) }{2\delta }\right) \right) ^{-1}g\left( x\right)
dx,
\end{equation*}%
where $\left\vert g\left( x\right) \right\vert \leq M.$ Conversely, since
the function $\left( \mathrm{\cosh }\left( \pi \left( x-iy\right) /2\delta
\right) \right) ^{-1}$ is analytic in the strip $\left\vert y\right\vert
<\delta ,x\in \mathbb{R}$ and 
\begin{equation*}
\func{Re}\left( \left( \cosh \left( \frac{\pi \left( z-x\right) }{2\delta }%
\right) \right) ^{-1}\right) >0,
\end{equation*}%
\begin{equation*}
\frac{1}{2\delta }\int_{\mathbb{R}}\left( \cosh \left( \frac{\pi \left(
z-x\right) }{2\delta }\right) \right) ^{-1}dx=1
\end{equation*}%
then the function 
\begin{equation*}
f\left( x+iy\right) =\frac{1}{2\delta }\int_{\mathbb{R}}\left( \cosh \left( 
\frac{\pi \left( x+iy-s\right) }{2\delta }\right) \right) ^{-1}g\left(
s\right) ds
\end{equation*}%
is analytic in the strip $\left\vert y\right\vert <\delta ,x\in \mathbb{R}$
and $\left\vert \func{Re}f\left( x+iy\right) \right\vert \leq M$ if $\mathrm{%
ess}\sup \left\vert g\left( x\right) \right\vert \leq M$. Observe that the
function 
\begin{equation*}
f\left( z\right) =\frac{1}{2i}\ln \left( \frac{e^{\pi z/\left( 2\delta
\right) }+i}{1+ie^{\pi z/\left( 2\delta \right) }}\right) \in A_{\infty
,\delta }UM
\end{equation*}%
is not bounded in the strip $\left\{ z|z=x+iy,\left\vert y\right\vert
<\delta ,x\in \mathbb{R}\right\} .$ It means that the set of functions $%
A_{\infty ,\delta }M$ which are bounded and analytically extendable into the
same strip is a proper subset of $A_{\infty ,\delta }UM$.

Similarly, in the multidimensional settings, the set $A_{\infty ,\mathbf{%
\delta }}UM,\mathbf{\delta }=\left( \delta _{1},\cdot \cdot \cdot ,\delta
_{n}\right) $ of functions $f\left( \mathbf{z}\right) ,\mathbf{z}=\left(
z_{1},\cdot \cdot \cdot ,z_{n}\right) \in \mathbb{C}^{n}$ which are
analytically extendable into the tube $\Omega \left( \mathbf{\delta }\right)
:=\left\{ \mathbf{z|z=}\left( z_{1},\cdot \cdot \cdot ,z_{n}\right) \mathbf{%
\in }\mathbb{C}^{n},\left\vert \func{Im}z_{k}\right\vert \leq \delta
_{k},1\leq k\leq n\right\} $%
\begin{equation*}
\Omega \left( \mathbf{\delta }\right) =\Omega \left( \delta _{1},\cdot \cdot
\cdot ,\delta _{n}\right) :=\left\{ \mathbf{z|z=}\left( z_{1},\cdot \cdot
\cdot ,z_{n}\right) \mathbf{\in }\mathbb{C}^{n},\left\vert \func{Im}%
z_{k}\right\vert \leq \delta _{k},1\leq k\leq n\right\}
\end{equation*}%
and such that $\left\vert \func{Re}f\left( \mathbf{z}\right) \right\vert
\leq M$, admits representation 
\begin{equation*}
f\left( \mathbf{z}\right) =\left( K\ast g\right) \left( \mathbf{z}\right)
:=\int_{\mathbb{R}^{n}}K\left( \mathbf{z-x}\right) g\left( \mathbf{x}\right)
d\mathbf{x},
\end{equation*}%
where%
\begin{equation}
K\left( \mathbf{z}\right) :=\dprod\limits_{k=1}^{n}K_{\left( k\right)
}\left( z_{k}\right)  \label{kernel 1}
\end{equation}%
and%
\begin{equation*}
K_{\left( k\right) }\left( x_{k}\right) :=\frac{1}{2\delta _{k}}\left( \cosh
\left( \frac{\pi z_{k}}{2\delta _{k}}\right) \right) ^{-1}.
\end{equation*}%
Let 
\begin{equation*}
E\left( f,W_{\mathbf{a}}\left( \mathbb{R}^{n}\right) ,L_{p}\left( \mathbb{R}%
^{n}\right) \right) :=\inf \left\{ \left\Vert f-g\right\Vert _{L_{p}\left( 
\mathbb{R}^{n}\right) }\left\vert g\in W_{\mathbf{a}}\left( \mathbb{R}%
^{n}\right) \right. \right\}
\end{equation*}%
be the best approximation of $f(z)$ by the space $W_{\mathbf{a}}\left( 
\mathbb{R}^{n}\right) $ in $L_{p}\left( \mathbb{R}^{n}\right) $ and 
\begin{equation*}
E\left( \mathcal{K},W_{\mathbf{a}}\left( \mathbb{R}^{n}\right) ,L_{p}\left( 
\mathbb{R}^{n}\right) \right) :=\sup \left\{ E\left( f,W_{\mathbf{a}}\left( 
\mathbb{R}^{n}\right) ,L_{p}\left( \mathbb{R}^{n}\right) \right) \left\vert
f\in \mathcal{K}\right. \right\} .
\end{equation*}%
be the best approximation of the function class $\mathcal{K}\in \
L_{p}\left( \mathbb{R}^{n}\right) $ by $W_{\mathbf{a}}\left( \mathbb{R}%
^{n}\right) $ in $\ L_{p}\left( \mathbb{R}^{n}\right) .$

\begin{lemma}
\label{entire function} Let 
\begin{equation*}
\vartheta _{\mathbf{a}}\left( \mathbf{x}\right)
:=\dprod\limits_{k=1}^{n}\vartheta _{a_{k}}\left( x_{k}\right) ,
\end{equation*}%
where $\vartheta _{a_{k}}\left( x_{k}\right) $ is an integrable entire
function of exponential type $a_{k}>0.$ Let $g\left( \mathbf{y}\right) $ be
a bounded function, i.e. $\left\vert g\left( \mathbf{y}\right) \right\vert
\leq M,\forall \mathbf{y\in }\mathbb{R}^{n}$. Then the function 
\begin{equation*}
\varrho _{\mathbf{a}}\left( \mathbf{x}\right) :=\int_{\mathbb{R}%
^{n}}\vartheta _{\mathbf{a}}\left( \mathbf{x-y}\right) g\left( \mathbf{y}%
\right) d\mathbf{y}
\end{equation*}%
is of exponential type $\mathbf{a}=\left( a_{1},\cdot \cdot \cdot
,a_{n}\right) .$
\end{lemma}

\begin{proof}
It follows from the definition that it is sufficient to show that 
\begin{equation*}
\left\vert \varrho _{\mathbf{a}}\left( \mathbf{x}\right) \right\vert
=\left\vert \varrho \left( x_{1},\cdot \cdot \cdot ,x_{l},\cdot \cdot \cdot
,x_{n}\right) \right\vert \leq M_{l}e^{a_{l}\left\vert x_{l}\right\vert
},1\leq l\leq n
\end{equation*}%
for some absolute constant $M_{l}$ and fixed $x_{k},k\neq l$. Expanding $%
\vartheta _{a_{l}}\left( x_{l}-y_{l}\right) ,1\leq l\leq n$ into the power
series with respect to $x_{k}$ we get 
\begin{equation*}
\vartheta _{a_{l}}\left( x_{l}-y_{l}\right) =\sum_{s=0}^{\infty }\frac{%
\left( \vartheta _{a_{l}}\left( -y_{l}\right) \right) ^{\left( s\right) }}{l!%
}x_{l}^{s}.
\end{equation*}%
Hence%
\begin{equation*}
\left\vert \varrho _{\mathbf{a}}\left( \mathbf{x}\right) \right\vert
=\left\vert \int_{\mathbb{R}^{n}}\left( \dprod\limits_{k=1}^{n}\vartheta
_{a_{k}}\left( x_{k}-y_{k}\right) \right) g\left( y_{1},\cdot \cdot \cdot
,y_{n}\right) \dprod\limits_{k=1}^{n}dy_{k}\right\vert
\end{equation*}%
\begin{equation*}
=\left\vert \int_{\mathbb{R}^{n}}\left( \sum_{s=0}^{\infty }\frac{\left(
\vartheta _{a_{l}}\left( -y_{l}\right) \right) ^{\left( s\right) }}{s!}%
x_{l}^{s}\right) \left( \dprod\limits_{k=1,k\neq l}^{n}\vartheta
_{a_{k}}\left( x_{k}-y_{k}\right) \right) g\left( y_{1},\cdot \cdot \cdot
,y_{n}\right) \dprod\limits_{k=1}^{n}dy_{k}\right\vert
\end{equation*}%
\begin{equation*}
\leq M\int_{\mathbb{R}^{n}}\left( \sum_{s=0}^{\infty }\frac{\left\vert
\left( \vartheta _{a_{l}}\left( -y_{l}\right) \right) ^{\left( s\right)
}\right\vert }{s!}x_{l}^{s}\right) \left( \dprod\limits_{k=1,k\neq
l}^{n}\left\vert \vartheta _{a_{k}}\left( x_{k}-y_{k}\right) \right\vert
\right) \dprod\limits_{k=1}^{n}dy_{k}
\end{equation*}%
\begin{equation*}
\leq MK_{l}\int_{\mathbb{R}}\sum_{s=0}^{\infty }\frac{\left\vert \left(
\vartheta _{a_{l}}\left( y_{l}\right) \right) ^{\left( s\right) }\right\vert 
}{s!}x_{l}^{s}dy_{l},
\end{equation*}%
where 
\begin{equation*}
K_{l}:=\int_{\mathbb{R}^{n-1}}\dprod\limits_{k=1,k\neq l}^{n}\left\vert
\vartheta _{a_{k}}\left( x_{k}-y_{k}\right) \right\vert
\dprod\limits_{k=1,k\neq l}^{n}dy_{k}.
\end{equation*}%
Applying Bernstein's inequality \cite{timan} p.230%
\begin{equation*}
\int_{\mathbb{R}}\left\vert \left( \vartheta _{a}\left( y\right) \right)
^{\left( s\right) }\right\vert dy\leq a^{s}\int_{\mathbb{R}}\left\vert
\vartheta _{a}\left( y\right) \right\vert dy,s\in \mathbb{N},
\end{equation*}%
which is valid for any entire function $\vartheta _{a}\left( y\right) $ of
exponential type $a,$ we get 
\begin{equation*}
\left\vert \varrho _{\mathbf{a}}\left( \mathbf{x}\right) \right\vert \leq
MK_{l}P_{l}\sum_{s=0}^{\infty }\frac{a_{l}^{s}}{s!}\left\vert
x_{l}\right\vert ^{s}
\end{equation*}%
\begin{equation*}
=MK_{l}P_{l}e^{a_{l}\left\vert x_{l}\right\vert },
\end{equation*}%
where%
\begin{equation*}
P_{l}:=\int_{\mathbb{R}}\left\vert \vartheta _{a}\left( y\right) \right\vert
dy.
\end{equation*}%
It means that $\varrho _{\mathbf{a}}\left( \mathbf{x}\right) $ is of
exponential type $\mathbf{a}$.
\end{proof}

\begin{lemma}
\label{kernel approximation} 
\begin{equation*}
E\left( \left( K\ast g\right) \left( \mathbf{z}\right) ,W_{\mathbf{a}}\left( 
\mathbb{R}^{n}\right) ,L_{\infty }\left( \mathbb{R}^{n}\right) \right) \leq
M\| K- \vartheta _{\mathbf{a}} \|_{1}
\end{equation*}
\begin{equation*}
\leq 2Mn\dprod\limits_{k=1}^{n}\left( \frac{S_{a_{k}}}{2\delta _{k}}\right)
\exp \left( -\min \left\{ a_{k}\delta _{k}\left\vert 1\leq k\leq n\right.
\right\} \right),
\end{equation*}
and 
\begin{equation*}
E\left( \left( K\ast g\right) \left( \mathbf{z}\right) ,W_{\mathbf{a}}\left( 
\mathbb{R}^{n}\right) ,L_{1 }\left( \mathbb{R}^{n}\right) \right) \leq L\|
K- \vartheta _{\mathbf{a}} \|_{1}
\end{equation*}
\begin{equation*}
\leq 2Ln\dprod\limits_{k=1}^{n}\left( \frac{S_{a_{k}}}{2\delta _{k}}\right)
\exp \left( -\min \left\{ a_{k}\delta _{k}\left\vert 1\leq k\leq n\right.
\right\} \right),
\end{equation*}
where $K$ is defined by (\ref{kernel 1}), $|g| \leq M$, $g \in L_{1}(\mathbb{%
R}^{n})$and 
\begin{equation*}
\vartheta _{\mathbf{a}}\left( \mathbf{x}\right)
=\dprod\limits_{k=1}^{n}\vartheta _{a_{k}}\left( x_{k}\right) ,
\end{equation*}
where 
\begin{equation*}
\vartheta _{a_{k}}\left( x_{k}\right) :=2\delta _{k}\int_{0}^{a_{k}}\left(
\mu _{k}\left( \xi \right) -C_{k}\left( a_{k}-\xi \right) -C_{k}\left(
a_{k}+\xi \right) \right) \cos \xi x_{k}d\xi ,
\end{equation*}
\begin{equation*}
\mu _{k}\left( \xi \right) :=\left( \pi \cosh \left( \delta _{k}\xi \right)
\right) ^{-1},
\end{equation*}
\begin{equation*}
C_{k}\left( \xi \right) :=\sum_{s=0}^{\infty }\left( -1\right) ^{s}\mu
_{k}\left( \left( 2s+1\right) a_{k}+\xi \right) ,1\leq k\leq n.
\end{equation*}
and 
\begin{equation*}
S_{a_{k}}:=2\delta _{k}a_{k}\left( \frac{1}{2\delta _{k}}+\mu _{k}\left(
2a_{k}\right) +\mu _{k}\left( a_{k}\right) \right).
\end{equation*}
\end{lemma}

\begin{proof}
We prove Lemma just in the case $p=\infty $. The case $p=1$ follows in the
similar way. It is easy to check that $\vartheta _{a_{k}}\left( x_{k}\right) 
$ is a function of exponential type $a_{k}$ and $\vartheta _{a_{k}}\left(
x_{k}\right) \in L_{2}\left( \mathbb{R}\right) .$ Hence $\vartheta _{\mathbf{%
a}}\left( \mathbf{x}\right) \in L_{2}\left( \mathbb{R}^{n}\right) .$

Consider the set of functions $K\ast g$, where $\left\vert g\right\vert \leq
M$ and $\left\Vert g\right\Vert _{1}\leq L.$ Assume that $\vartheta _{%
\mathbf{a}}\in L_{2}\left( \mathbb{R}^{n}\right) $. In this case,%
\begin{equation*}
\left\Vert \int_{\mathbb{R}^{n}}\vartheta _{\mathbf{a}}\ast g\right\Vert
_{2}\leq \left\Vert \vartheta _{\mathbf{a}}\right\Vert _{2}\left\Vert
g\right\Vert _{1}\leq L\left\Vert \vartheta _{\mathbf{a}}\right\Vert _{2}.
\end{equation*}%
Hence, by Lemma \ref{entire function}, $\left( \vartheta _{\mathbf{a}}%
\mathbf{\ast }g\right) \in W_{\mathbf{a}}\left( \mathbb{R}^{n}\right) $ and 
\begin{equation*}
E\left( \left( K\ast g\right) ,W_{\mathbf{a}}\left( \mathbb{R}^{n}\right)
,L_{\infty }\left( \mathbb{R}^{n}\right) \right) \leq \sup \left\{
\left\Vert K\ast g-\vartheta _{\mathbf{a}}\ast g\right\Vert _{\infty
}\left\vert \left\vert g\right\vert \leq M,\left\Vert g\right\Vert _{1}\leq
L\right. \right\}
\end{equation*}%
\begin{equation}
\leq M\left\Vert K-\vartheta _{\mathbf{a}}\right\Vert _{1}.  \label{k-t1}
\end{equation}%
Observe that for any complex numbers $\rho _{1,m}$ and $\rho _{2,m},1\leq
m\leq n$ we have%
\begin{equation}
\dprod\limits_{m=1}^{n}\rho _{1,m}-\dprod\limits_{m=1}^{n}\rho
_{2,m}=\sum_{m=1}^{n}\left( \rho _{1,m}-\rho _{2,m}\right)
\dprod\limits_{r=1}^{m-1}\rho _{2,r}\dprod\limits_{r=m+1}^{n}\rho _{1,r}.
\label{product1}
\end{equation}%
It easy to check that 
\begin{equation}
\sup_{z_{k}\in \mathbb{R}}\frac{1}{2\delta _{k}}\left( \cosh \left( \frac{%
\pi z_{k}}{2\delta _{k}}\right) \right) ^{-1}\leq \frac{1}{2\delta _{k}},
\label{sup prod}
\end{equation}%
\begin{equation*}
0\leq \mu _{k}\left( \xi \right) \leq \frac{1}{\pi },\xi \in \mathbb{R},
\end{equation*}%
\begin{equation*}
\left\vert C_{k}\left( a_{k}+\xi \right) \right\vert \leq \mu _{k}\left(
2a_{k}\right) ,0\leq \xi \leq a_{k}
\end{equation*}%
and%
\begin{equation*}
\left\vert C_{k}\left( a_{k}-\xi \right) \right\vert \leq \mu _{k}\left(
a_{k}\right) ,0\leq \xi \leq a_{k}.
\end{equation*}%
Hence%
\begin{equation}
\left\vert \vartheta _{a_{k}}\left( x_{k}\right) \right\vert \leq S_{a_{k}}.
\label{vartheta1}
\end{equation}%
Comparing (\ref{k-t1}) - (\ref{vartheta1}) we get 
\begin{equation*}
E\left( \left( K\ast g\right) ,W_{\mathbf{a}}\left( \mathbb{R}^{n}\right)
,L_{\infty }\left( \mathbb{R}^{n}\right) \right)
\end{equation*}%
\begin{equation*}
\leq Mn\dprod\limits_{k=1}^{n}\left( \frac{S_{a_{k}}}{2\delta _{k}}\right)
\max \left\{ \left\Vert \left( \cosh \left( \delta _{k}\xi \right) \right)
^{-1}-\vartheta _{a_{k}}\left( \xi \right) \right\Vert _{1}\left\vert 1\leq
k\leq n\right. \right\}
\end{equation*}%
It is known (see, e.g. \cite{timan}, p. 320) that for sufficiently large $%
a_{k},$ 
\begin{equation*}
\left\Vert \left( \cosh \left( \delta _{k}\xi \right) \right)
^{-1}-\vartheta _{a_{k}}\left( \xi \right) \right\Vert _{1}\leq \frac{%
2\delta _{k}}{\pi }\sum_{s=0}^{\infty }\frac{\left( -1\right) ^{s}}{\left(
2s+1\right) \cosh \left( \left( 2s+1\right) a_{k}\delta _{k}\right) }.
\end{equation*}%
Observe that 
\begin{equation*}
\sum_{k=0}^{\infty }\frac{\left( -1\right) ^{k}}{\left( 2k+1\right) \cosh
\left( \left( 2k+1\right) \sigma \delta \right) }<\frac{1}{\cosh \left(
\sigma \delta \right) }
\end{equation*}%
\begin{equation*}
=\frac{2}{e^{\sigma \delta }+e^{-\sigma \delta }}<2e^{-\sigma \delta }.
\end{equation*}%
Hence, 
\begin{equation*}
E\left( \left( K\ast g\right) ,W_{\mathbf{a}}\left( \mathbb{R}^{n}\right)
,L_{\infty }\left( \mathbb{R}^{n}\right) \right)
\end{equation*}%
\begin{equation*}
\leq 2Mn\dprod\limits_{k=1}^{n}\left( \frac{S_{a_{k}}}{2\delta _{k}}\right)
\exp \left( -\min \left\{ a_{k}\delta _{k}\left\vert 1\leq k\leq n\right.
\right\} \right) .
\end{equation*}
\end{proof}

\begin{lemma}
\label{q density} \bigskip Let $B\left( \mathbb{R}^{n}\right) $ be the set
of bounded measurable functions on $\mathbb{R}^{n}$ then for any density
function $p_{t}\mathbf{\in }B\left( \mathbb{R}^{n}\right) $ we have 
\begin{equation*}
p_{t}=\mathbf{F}^{-1}\left( e^{-t\psi \left( \mathbf{\cdot }\right) }\right) 
\mathbf{\in }\dbigcap\limits_{q=1}^{\infty }L_{q}\left( \mathbb{R}%
^{n}\right) .
\end{equation*}
\end{lemma}

\begin{proof}
\bigskip The symmetric rearrangement of a measurable set $A\in \mathbb{R}%
^{n} $ is defined as $A^{\sim }=\left\{ \mathbf{x}\in \mathbb{R}%
^{n}\left\vert \omega _{n}\left\langle \mathbf{x},\mathbf{x}\right\rangle
^{n/2}<\mathrm{Vol}_{n}A\right. \right\} ,$ where $\omega _{n}$ is the
volume of the unit ball in $\mathbb{R}^{n}$. The symmetric rearrangement $%
f^{\sim }$ of a measurable function $f\geq 0$ is defined as 
\begin{equation*}
f^{\sim }\left( \mathbf{x}\right) =\int_{0}^{\infty }\chi _{\left\{ \mathbf{y%
}:f\left( \mathbf{y}\right) >t\right\} ^{\sim }}\left( \mathbf{x}\right) dt
\end{equation*}%
Since 
\begin{equation*}
\int_{\mathbb{R}^{n}}p_{t}\left( \mathbf{x}\right) d\mathbf{x}=1,p_{t}\left( 
\mathbf{x}\right) \geq 0,\forall \mathbf{x\in }\mathbb{R}^{n}
\end{equation*}%
then $p_{t}\mathbf{\in }B\left( \mathbb{R}^{n}\right) \cap L_{1}\left( 
\mathbb{R}^{n}\right) $ and $\lim_{\left\langle \mathbf{x,x}\right\rangle
\rightarrow \infty }p_{t}^{\sim }\left( \mathbf{x}\right) =0$. It means that
there is a set $B\subset \mathbb{R}^{n}$ such that $\mathrm{Vol}_{n}B<\infty 
$ and $p_{t}^{\sim }\left( \mathbf{x}\right) \leq 1,\forall \mathbf{x\in }%
\mathbb{R}^{n}\setminus B.$ Hence for any $1\leq q<\infty $ we get 
\begin{equation*}
\int_{\mathbb{R}^{n}}p_{t}^{q}\left( \mathbf{x}\right) d\mathbf{x}=\int_{%
\mathbb{R}^{n}}\left( p_{t}^{\sim }\left( \mathbf{x}\right) \right) ^{q}d%
\mathbf{x=}\left( \int_{B}+\int_{\mathbb{R}^{n}\setminus B}\right) \left(
p_{t}^{\sim }\left( \mathbf{x}\right) \right) ^{q}d\mathbf{x}
\end{equation*}%
\begin{equation*}
\leq C^{q}\mathrm{Vol}_{n}B+\int_{\mathbb{R}^{n}\setminus B}p_{t}^{\sim
}\left( \mathbf{x}\right) d\mathbf{x\leq }C^{q}\mathrm{Vol}_{n}B+1.
\end{equation*}%
Therefore, $p_{t}\left( \mathbf{x}\right) \in B\left( \mathbb{R}^{n}\right)
\cap L_{1}\left( \mathbb{R}^{n}\right) \cap L_{q}\left( \mathbb{R}%
^{n}\right) $ and applying Plancherel's theorem we get $p_{t}\left( \mathbf{x%
}\right) =\mathbf{F}^{-1}\left( e^{-t\psi \left( \mathbf{\cdot }\right)
}\right) \left( \mathbf{x}\right) .$
\end{proof}

\begin{lemma}
\label{uniform approximation} \bigskip In our notations 
\begin{equation*}
\left\Vert e^{-t\psi \left( \mathbf{\cdot }\right) }- \sum_{\mathbf{m}\in 
\mathbb{Z}^{n}}\exp \left( -t\psi \left( \pi \mathrm{A}\mathbf{m}^{\mathrm{T}%
}\right) \right) J_{\mathbf{m},\lambda _{a}}\left( \cdot \right) \right\Vert
_{\infty }
\end{equation*}%
\begin{equation*}
\leq 2\left( 1+2\times 2.834^{n}\right) Mn\dprod\limits_{k=1}^{n}\left( 
\frac{S_{a_{k}}}{2\delta _{k}}\right) \exp \left( -\min \left\{ a_{k}\delta
_{k}\left\vert 1\leq k\leq n\right. \right\} \right).
\end{equation*}
\end{lemma}

\begin{proof}
Since for any $\rho _{\mathbf{a}}\in W_{\mathbf{a}}\left( \mathbb{R}%
^{n}\right) $ we have 
\begin{equation*}
\rho _{\mathbf{a}}\left( \cdot \right) =\sum_{\mathbf{m}\in \mathbb{Z}%
^{n}}\rho _{\mathbf{a}}\left( \pi \mathrm{A}\mathbf{m}^{\mathrm{T}}\right)
J_{\mathbf{m},\lambda _{a}}\left( \cdot \right) .
\end{equation*}%
then applying Lemma \ref{kernel approximation} and Corollary \ref{norm 1} we
get 
\begin{equation*}
\left\Vert e^{-t\psi \left( \mathbf{\cdot }\right) }-\sum_{\mathbf{m}\in 
\mathbb{Z}^{n}}\exp \left( -t\psi \left( \pi \mathrm{A}\mathbf{m}^{\mathrm{T}%
}\right) \right) J_{\mathbf{m},\lambda _{a}}\left( \cdot \right) \right\Vert
_{\infty }
\end{equation*}%
\begin{equation*}
=\left\Vert e^{-t\psi \left( \mathbf{\cdot }\right) }-\rho _{\mathbf{a}%
}\left( \cdot \right) +\rho _{\mathbf{a}}\left( \cdot \right) -\sum_{\mathbf{%
m}\in \mathbb{Z}^{n}}\exp \left( -t\psi \left( \pi \mathrm{A}\mathbf{m}^{%
\mathrm{T}}\right) \right) J_{\mathbf{m},\lambda _{a}}\left( \cdot \right)
\right\Vert _{\infty }
\end{equation*}%
\begin{equation*}
\leq E\left( \left( K\ast g\right) ,W_{\mathbf{a}}\left( \mathbb{R}%
^{n}\right) ,L_{\infty }\left( \mathbb{R}^{n}\right) \right)
\end{equation*}%
\begin{equation*}
+\left\Vert \rho _{\mathbf{a}}\left( \cdot \right) -\sum_{\mathbf{m}\in 
\mathbb{Z}^{n}}\exp \left( -t\psi \left( \pi \mathrm{A}\mathbf{m}^{\mathrm{T}%
}\right) \right) J_{\mathbf{m},\lambda _{a}}\left( \cdot \right) \right\Vert
_{\infty }
\end{equation*}%
\begin{equation*}
\leq E\left( \left( K\ast g\right) ,W_{\mathbf{a}}\left( \mathbb{R}%
^{n}\right) ,L_{\infty }\left( \mathbb{R}^{n}\right) \right) \left(
1+\left\Vert \mathbf{P}_{\lambda _{a}}\left\vert C\left( \mathbb{R}\right)
\rightarrow C\left( \mathbb{R}\right) \right. \right\Vert \right)
\end{equation*}%
\begin{equation*}
\leq 2\left( 1+2\times 2.834^{n}\right) Mn\dprod\limits_{k=1}^{n}\left( 
\frac{S_{a_{k}}}{2\delta _{k}}\right) \exp \left( -\min \left\{ a_{k}\delta
_{k}\left\vert 1\leq k\leq n\right. \right\} \right) .
\end{equation*}
\end{proof}

\begin{lemma}
\label{ext1} \bigskip \bigskip Assume that the characteristic function $\Phi
\left( \mathbf{x,}t\right) $ :$=\exp \left( -t\psi \left( \mathbf{x}\right)
\right) $ admits an analytic extension into the tube $\Omega \left( \mathbf{%
\delta }\right) $ and $\Phi \left( \mathbf{x}\right) \in L_{1}\left( \mathbb{%
R}^{n}\right) \cap L_{2}\left( \mathbb{R}^{n}\right) $ then for any $\mathbf{%
\alpha }\in \mathbb{R}^{n},\left\Vert \mathbf{\alpha }\right\Vert _{\infty
}<\left\Vert \mathbf{\delta }\right\Vert _{\infty }$ we have%
\begin{equation*}
p_{t}\left( \cdot \right) =\frac{1}{2}\left( 2\pi \right) ^{-n}\left( \cosh
\left( \left\langle \mathbf{\alpha },\cdot \right\rangle \right) \right)
^{-1}\int_{\mathbb{R}^{n}}\exp \left( i\left\langle \mathbf{z},\cdot
\right\rangle -t\left( \psi \left( \mathbf{z+}i\mathbf{\alpha }\right) +\psi
\left( \mathbf{z-}i\mathbf{\alpha }\right) \right) \right) d\mathbf{z.}
\end{equation*}
\end{lemma}

\begin{proof}
Since the characteristic function $\Phi \left( \mathbf{x,}t\right) $ :$=\exp
\left( -t\psi \left( \mathbf{x}\right) \right) $ admits an analytic
extension into the tube $\Omega \left( \mathbf{\delta }\right) $ and $\Phi
\left( \mathbf{x}\right) \in L_{1}\left( \mathbb{R}^{n}\right) \cap
L_{2}\left( \mathbb{R}^{n}\right) $ then%
\begin{equation*}
\lim_{\left\langle \mathbf{x,x}\right\rangle \rightarrow \infty }\Phi \left( 
\mathbf{x}\right) =0
\end{equation*}%
and applying Caychy's theorem we get 
\begin{equation*}
p_{t}\left( \mathbf{\cdot }\right) =\left( 2\pi \right) ^{-n}\int_{\mathbb{R}%
^{n}}\exp \left( i\left\langle \mathbf{x},\cdot \right\rangle -t\psi \left( 
\mathbf{x}\right) \right) d\mathbf{x}
\end{equation*}%
\begin{equation*}
=\left( 2\pi \right) ^{-n}\int_{\mathbb{R}^{n}+i\mathbf{\alpha }}\exp \left(
i\left\langle \mathbf{x},\cdot \right\rangle -t\psi \left( \mathbf{x}\right)
\right) d\mathbf{x}
\end{equation*}%
\begin{equation*}
=\left( 2\pi \right) ^{-n}\int_{\mathbb{R}^{n}}\exp \left( i\left\langle 
\mathbf{z+}i\mathbf{\alpha },\cdot \right\rangle -t\psi \left( \mathbf{z+}i%
\mathbf{\alpha }\right) \right) d\mathbf{z}
\end{equation*}%
\begin{equation*}
=\left( 2\pi \right) ^{-n}\exp \left( -\left\langle \mathbf{\alpha },\cdot
\right\rangle \right) \int_{\mathbb{R}^{n}}\exp \left( i\left\langle \mathbf{%
z},\cdot \right\rangle -t\psi \left( \mathbf{z+}i\mathbf{\alpha }\right)
\right) d\mathbf{z,}
\end{equation*}%
where $\left\Vert \mathbf{\alpha }\right\Vert _{\infty }<\left\Vert \mathbf{%
\delta }\right\Vert _{\infty }$. Hence%
\begin{equation*}
p_{t}\left( \mathbf{\cdot }\right) \exp \left( -\left\langle \mathbf{\alpha }%
,\cdot \right\rangle \right) =\left( 2\pi \right) ^{-n}\int_{\mathbb{R}%
^{n}}\exp \left( i\left\langle \mathbf{z},\cdot \right\rangle -t\psi \left( 
\mathbf{z+}i\mathbf{\alpha }\right) \right) d\mathbf{z.}
\end{equation*}%
Similarly, 
\begin{equation*}
p_{t}\left( \mathbf{\cdot }\right) \exp \left( \left\langle \mathbf{\alpha }%
,\cdot \right\rangle \right) =\left( 2\pi \right) ^{-n}\int_{\mathbb{R}%
^{n}}\exp \left( i\left\langle \mathbf{z},\cdot \right\rangle -t\psi \left( 
\mathbf{z-}i\mathbf{\alpha }\right) \right) d\mathbf{z.}
\end{equation*}%
It means that 
\begin{equation*}
p_{t}\left( \cdot \right) =\frac{1}{2}\left( 2\pi \right) ^{-n}\left( \cosh
\left( \left\langle \mathbf{\alpha },\cdot \right\rangle \right) \right)
^{-1}\int_{\mathbb{R}^{n}}\exp \left( i\left\langle \mathbf{z},\cdot
\right\rangle -t\left( \psi \left( \mathbf{z+}i\mathbf{\alpha }\right) +\psi
\left( \mathbf{z-}i\mathbf{\alpha }\right) \right) \right) d\mathbf{z.}
\end{equation*}
\end{proof}

\begin{lemma}
Assume that the characteristic function $\Phi \left( \mathbf{x,}T\right) $ :$%
=\exp \left( -T\psi \left( \mathbf{x}\right) \right) \in L_{1}\left( \mathbb{%
R}^{n}\right) \cap L_{2}\left( \mathbb{R}^{n}\right) $ and admits an
analytic extension into the tube $\Omega \left( \mathbf{\delta }\right) $.
Let 
\begin{equation*}
p_{T,\mathbf{a}}^{\ast }\left( \mathbf{\cdot }\right) :=\left( 2\pi \right)
^{-n}\int_{\mathbb{R}^{n}}\exp \left( i\left\langle \mathbf{z},\cdot
\right\rangle \right) g_{\mathbf{a}}\left( \mathbf{z,}T\right) d\mathbf{z,}
\end{equation*}%
where $g_{\mathbf{a}}\left( \mathbf{z,}T\right) \in W_{\mathbf{a}}\left( 
\mathbb{R}^{n}\right) $ interpolates 
\begin{equation*}
\Xi \left( \mathbf{\cdot ,}T\right) :=\exp \left( -t\left( \psi \left( 
\mathbf{z+}i\mathbf{\alpha }\right) +\psi \left( \mathbf{z-}i\mathbf{\alpha }%
\right) \right) \right)
\end{equation*}%
at the points $\mathbf{z}_{m}=\mathrm{A}\mathbf{m}^{\mathrm{T}},\mathbf{m}%
\in \mathbb{Z}^{n}$. Then 
\begin{equation*}
\left\Vert p_{T}\left( \mathbf{\cdot }\right) -p_{T,\mathbf{a}}^{\ast
}\left( \mathbf{\cdot }\right) \right\Vert _{1}
\end{equation*}%
\begin{equation*}
\leq \left( 2\pi \right) ^{-n}\dprod\limits_{k=1}^{n}\alpha _{k}^{-1}\mathrm{%
Vol}\left( A\right) \left\Vert \Xi \left( \mathbf{z,}T\right) -g_{\mathbf{a}%
}\left( \mathbf{z,}T\right) \right\Vert _{\infty }+\varepsilon ,
\end{equation*}%
where $A\subset \mathbb{R}^{n}$ is such that 
\begin{equation*}
\left( 2\pi \right) ^{-n}\dprod\limits_{k=1}^{n}\alpha _{k}^{-1}\left( \int_{%
\mathbb{R}^{n}\setminus A}\left\vert \Xi \left( \mathbf{z,}T\right)
\right\vert d\mathbf{z+}\int_{\mathbb{R}^{n}\setminus A}\sum_{\mathbf{m}\in 
\mathbb{Z}^{n}}\left\vert \Xi \left( \pi \mathrm{A}\mathbf{m,}T\right)
\right\vert +\int_{\mathbb{R}^{n}\setminus A}\left\vert J_{\mathbf{m}%
,\lambda _{\mathbf{a}}}\left( \mathbf{x}\right) \right\vert d\mathbf{x}%
\right) \mathbf{\leq }\varepsilon
\end{equation*}%
and, in our notations, 
\begin{equation*}
\left\Vert \Xi \left( \mathbf{z,}T\right) -g_{\mathbf{a}}\left( \mathbf{z,}%
T\right) \right\Vert _{\infty }
\end{equation*}%
\begin{equation*}
\leq 2\left( 1+2\times 2.834^{n}\right) Mn\dprod\limits_{k=1}^{n}\left( 
\frac{S_{a_{k}}}{2\left( \delta _{k}-\alpha _{k}\right) }\right) \exp \left(
-\min \left\{ a_{k}\left( \delta _{k}-\alpha _{k}\right) \left\vert 1\leq
k\leq n\right. \right\} \right)
\end{equation*}
\end{lemma}

\begin{proof}
Applying Lemma \ref{ext1} we get 
\begin{equation*}
\left\Vert p_{T}\left( \mathbf{\cdot }\right) -p_{T,\mathbf{a}}^{\ast
}\left( \mathbf{\cdot }\right) \right\Vert _{1}
\end{equation*}%
\begin{equation*}
=\frac{1}{2}\left( 2\pi \right) ^{-n}\left\Vert \frac{1}{\cosh \left(
\left\langle \mathbf{\alpha },\cdot \right\rangle \right) }\int_{\mathbb{R}%
^{n}}\exp \left( i\left\langle \mathbf{z},\cdot \right\rangle \right) \left(
\Xi \left( \mathbf{z,}T\right) -g_{\mathbf{a}}\left( \mathbf{z,}T,A\right)
\right) d\mathbf{z}\right\Vert _{1}
\end{equation*}%
\begin{equation*}
\leq \frac{1}{2}\left( 2\pi \right) ^{-n}\left\Vert \frac{1}{\cosh \left(
\left\langle \mathbf{\alpha },\cdot \right\rangle \right) }\int_{\mathbb{R}%
^{n}}\left\vert \Xi \left( \mathbf{z,}T\right) -g_{\mathbf{a}}\left( \mathbf{%
z,}T,A\right) \right\vert d\mathbf{z}\right\Vert _{1}
\end{equation*}%
\begin{equation*}
\leq \frac{1}{2}\left( 2\pi \right) ^{-n}\dprod\limits_{k=1}^{n}\alpha
_{k}^{-1}\int_{\mathbb{R}^{n}}\left\vert \Xi \left( \mathbf{z,}T\right) -g_{%
\mathbf{a}}\left( \mathbf{z,}T,A\right) \right\vert d\mathbf{z}
\end{equation*}%
\begin{equation*}
=\frac{1}{2}\left( 2\pi \right) ^{-n}\dprod\limits_{k=1}^{n}\alpha
_{k}^{-1}\left( \int_{A}+\int_{\mathbb{R}^{n}\setminus A}\right) \left\vert
\Xi \left( \mathbf{z,}T\right) -g_{\mathbf{a}}\left( \mathbf{z,}T\right)
\right\vert d\mathbf{z}
\end{equation*}%
\begin{equation*}
\leq \frac{1}{2}\left( 2\pi \right) ^{-n}\dprod\limits_{k=1}^{n}\alpha
_{k}^{-1}\mathrm{Vol}\left( A\right) \left\Vert \Xi \left( \mathbf{z,}%
T\right) -g_{\mathbf{a}}\left( \mathbf{z,}T\right) \right\Vert _{\infty }
\end{equation*}%
\begin{equation*}
+\frac{1}{2}\left( 2\pi \right) ^{-n}\dprod\limits_{k=1}^{n}\alpha
_{k}^{-1}\left( \int_{\mathbb{R}^{n}\setminus A}\left\vert \Xi \left( 
\mathbf{z,}T\right) \right\vert d\mathbf{z+}\int_{\mathbb{R}^{n}\setminus
A}\left\vert g_{\mathbf{a}}\left( \mathbf{z,}T\right) \right\vert d\mathbf{z}%
\right) \mathbf{,}
\end{equation*}%
where $C$ is some absolute constant. Since $\Phi \left( \mathbf{x,}T\right)
\in L_{1}\left( \mathbb{R}^{n}\right) \cap L_{2}\left( \mathbb{R}^{n}\right) 
$ then $\Xi \left( \mathbf{x,}T\right) \in L_{1}\left( \mathbb{R}^{n}\right)
\cap L_{2}\left( \mathbb{R}^{n}\right) $. Hence for any $\varepsilon >0$
there is such $A:=A_{\varepsilon }$ that 
\begin{equation*}
\int_{\mathbb{R}^{n}\setminus A_{\varepsilon }}\left\vert \Xi \left( \mathbf{%
z,}T\right) \right\vert d\mathbf{z\leq }\varepsilon
\end{equation*}%
and%
\begin{equation*}
\int_{\mathbb{R}^{n}\setminus A_{\varepsilon }}\left\vert g_{\mathbf{a}%
}\left( \mathbf{z,}T\right) \right\vert d\mathbf{z\leq }\sum_{\mathbf{m}\in 
\mathbb{Z}^{n}}\left\vert \Xi \left( \pi \mathrm{A}\mathbf{m,}T\right)
\right\vert \int_{\mathbb{R}^{n}\setminus A_{\varepsilon }}\left\vert J_{%
\mathbf{m},\lambda _{\mathbf{a}}}\left( \mathbf{x}\right) \right\vert d%
\mathbf{x.}
\end{equation*}

Finally, we apply Lemma \ref{uniform approximation}.
\end{proof}

Observe that a similar estimate can be obtained in the case $\left\Vert
p_{T}\left( \mathbf{\cdot }\right) -p_{T,\mathbf{a}}^{\ast }\left( \mathbf{%
\cdot }\right) \right\Vert _{\infty }.$ Finally, examples 1-3 show that the
function $\Xi \left( \pi \mathrm{A}\cdot \mathbf{,}T\right) $ is rapidly
decaying. Consequently, we can effectively truncate the Fourier series in
the representation of $p_{T,\mathbf{a}}^{\ast }\left( \mathbf{\cdot }\right) 
$ to reduce significantly the number of point evaluation.

\section{\protect\bigskip Appendix. characteristic exponents and density
functions}

\label{repr p}

\bigskip

\bigskip Let $\left( \Omega ,\mathcal{F},\mathrm{P}\right) $be a probability
space. Let $\mathcal{B}\left( \mathbb{R}^{n}\right) $be the collection of
all \textit{Borel sets} on $\mathbb{R}^{n}$which is the $\sigma -$algebra
generated by all open sets in $\mathbb{R}^{n},$i.e., the smallest $\sigma -$%
algebra that contains all open sets in $\mathbb{R}^{n}.$A real valued
function is called \textit{measurable} (Borel measurable) if it is $\mathcal{%
B}\left( \mathbb{R}^{n}\right) $measurable. A mapping $\mathbf{X:}$$\Omega $$%
\longrightarrow \mathbb{R}^{n}$is an $\mathbb{R}^{n}-$valued \textit{random
variable} if it is $\mathcal{F}$-measurable, i.e., for any $B\in \mathcal{B}%
\left( \mathbb{R}^{n}\right) $we have $\left\{ \omega |\mathbf{X}(\omega
)\in B\right\} \in \mathcal{F}.$A \textit{stochastic process} $\mathbf{X}=\{%
\mathbf{X}_{t}\}_{t\in \mathbb{R}_{+}}$is a one-parametric family of random
variables on a common probability space $\left( \Omega ,\mathcal{F},\mathrm{P%
}\right) $. The\textit{\ trajectory} of the process $\mathbf{X}$ is a map 
\begin{equation*}
\begin{array}{ccc}
\mathbb{R}_{+} & \longrightarrow & \mathbb{R}^{n} \\ 
t & \longmapsto & \mathbf{X}_{t}\left( \omega \right) ,%
\end{array}%
\end{equation*}%
where $\omega \in \Omega $ and $\mathbf{X}_{t}=\left( X_{1t},\cdot \cdot
\cdot ,X_{nt}\right) $.

$\mathbf{X}=\{\mathbf{X}_{t}\}_{t\in \mathbb{R}_{+}}$ is called a \textit{L%
\'{e}vy process} (process with stationary independent increments) if

\begin{enumerate}
\item The random variables $\mathbf{X}_{t_{0}},\mathbf{X}_{t_{1}}-\mathbf{X}%
_{t_{0}},\cdots ,\mathbf{X}_{t_{m}}-\mathbf{X}_{t_{m-1}}$, for any $0\leq
t_{0}<t_{1}<\cdots <t_{m}$ and $m\in \mathbb{N}$ are independent (\textit{%
independent increment property}).

\item $\mathbf{X}_{0}=\mathbf{0}$ a.s.

\item The distribution of $\mathbf{X}_{t+\tau }-\mathbf{X}_{t}$ is
independent of $\tau $ (\textit{temporal homogeneity or stationary
increments property}).

\item It is \textit{stochastically continuous}, i.e. 
\begin{equation*}
\lim_{\tau \rightarrow t}\mathrm{P}\left[ |\mathbf{X}_{\tau }-\mathbf{X}%
_{t}|>\epsilon \right] =0
\end{equation*}%
for any $\epsilon >0$ and $t\geq 0$.

\item There is $\Omega _{0}\in \mathcal{F}$ with $\mathrm{P}\left( \Omega
_{0}\right) =1$ such that, for any $\omega \in \Omega _{0},$ $\mathbf{X}%
_{t}\left( \omega \right) $ is right-continuous on $\mathbb{[}0,\infty )$
and has left limits on $\mathbb{(}0,\infty ).$
\end{enumerate}

A process satisfying ($1-4$) is called a \textit{L\'{e}vy process in law. An
additive process }is a stochastic process which satisfies ($1,2,4,5$) and an%
\textit{\ additive process in law} satisfies ($1,2,4$).

The \textit{convolution} $\mu =\mu _{1}\ast \mu _{2}$ of two distributions $%
\mu _{1}$ and $\mu _{2}$ on $\mathbb{R}^{n}$ is defined as 
\begin{equation*}
\mu \left( B\right) =\int_{\mathbb{R}^{n}\mathbb{\times R}^{n}}\chi _{B}(%
\mathbf{x}+\mathbf{y})\mu _{1}\left( d\mathbf{x}\right) \mu _{2}\left( d%
\mathbf{y}\right) <\infty ,
\end{equation*}%
where 
\begin{equation*}
\chi _{B}(\mathbf{x}):=\left\{ 
\begin{array}{cc}
1, & \mathbf{x\in }B\mathbf{,} \\ 
0, & \mathbf{x\notin }B%
\end{array}%
\right.
\end{equation*}%
is the \textit{characteristic function} of a Borel (Lebesgue) measurable set 
$B\subset \mathbb{R}^{n}$. A probability measure $\mu $ is called \textit{%
infinitely divisible} if for any $m\in \mathbb{N}$ there is a probability
measure $\mu _{\left( m\right) }$ such that 
\begin{equation*}
\mu =\underbrace{\mu _{\left( m\right) }\ast \cdot \cdot \cdot \ast \mu
_{\left( m\right) }}_{m}.
\end{equation*}%
Consider the set $\mathcal{L\ }$of L\'{e}vy process $\mathbf{X}=\{\mathbf{X}%
_{t}\}_{t\in \mathbb{R}_{+}}$ on a probability space $(\Omega ,\mathcal{F},%
\mathrm{P})$. For a finite measure $\mu $ on $\mathbb{R}^{n}$ (i.e., if $\mu
\left( \mathbb{R}^{n}\right) <\infty $) we define its Fourier transform as 
\begin{equation*}
\widehat{\mu }\left( \mathbf{y}\right) =\mathbf{F}\mu \left( \mathbf{y}%
\right) =\int_{\mathbb{R}^{n}}\exp \left( -i\left\langle \mathbf{x},\mathbf{y%
}\right\rangle \right) \mu \left( d\mathbf{x}\right)
\end{equation*}%
and its formal inverse%
\begin{equation*}
\mu \left( d\mathbf{x}\right) =\mathbf{F}^{-1}\widehat{\mu }\left( d\mathbf{x%
}\right) =\frac{1}{\left( 2\pi \right) ^{n}}\int_{\mathbb{R}^{n}}\exp \left(
i\left\langle \mathbf{x},\mathbf{y}\right\rangle \right) \widehat{\mu }%
\left( \mathbf{y}\right) d\mathbf{y}.
\end{equation*}%
It is known that if $\mu $ is infinitely divisible then there exists a
unique continuous function $\phi :\mathbb{R}^{n}\mathbb{\rightarrow C}$ such
that $\phi \left( \mathbf{0}\right) =0$ and $e^{\phi \left( \mathbf{y}%
\right) }=\widehat{\mu }\left( \mathbf{y}\right) .$ Hence, the \textit{%
characteristic function} of the distribution of $\mathbf{X}_{t}$ of any L%
\'{e}vy process can be represented in the form 
\begin{equation*}
\mathbb{E}\left[ \exp \left( \left\langle i\mathbf{x,X}_{t}\right\rangle
\right) \right] =e^{-t\psi \left( \mathbf{x}\right) },
\end{equation*}%
where $\mathbf{x}\in \mathbb{R}^{n}$, $t\in \mathbb{R}_{+}$ and the function 
$\psi \left( \mathbf{x}\right) $ \ is uniquely determined. This function is
called the \textit{characteristic exponent}. Vice versa, a L\'{e}vy process $%
\mathbf{X}=\{\mathbf{X}_{t}\}_{t\in \mathbb{R}_{+}}$ is determined uniquely
by its characteristic exponent $\psi \left( \mathbf{x}\right) $. In
particular, density function $p_{t}$ can be expressed as 
\begin{equation}
p_{t}\left( \cdot \right) =\frac{1}{\left( 2\pi \right) ^{n}}\int_{\mathbb{R}%
^{n}}\exp \left( i\left\langle \cdot \mathbf{,x}\right\rangle -t\psi \left( 
\mathbf{x}\right) \right) =\mathbf{F}^{-1}\left( \exp \left( -t\psi \left( 
\mathbf{x}\right) \right) \right) \left( \cdot \right) .  \label{density 1}
\end{equation}

We say that a matrix $A$ is \textit{nonnegative-definite} (or
positive-semidefinite) if $\left\langle \mathbf{x}^{\ast },A\mathbf{x}%
\right\rangle \geq 0$ for all $\mathbf{x}\in \mathbb{C}^{n}$ (or for all $%
\mathbf{x}\in \mathbb{R}$ for the real matrix), where $\mathbf{x}^{\ast }$
is the conjugate transpose.

The key role in our analysis plays the following classical result known as
the L\'{e}vy-Khintchine formula which gives a representation of
characteristic functions of all infinitely divisible distributions.

\begin{theorem}
Let $\mathbf{X}=\{\mathbf{X}_{t}\}_{t\in \mathbb{R}_{+}}$ be a L\'{e}vy
process on $\mathbb{R}^{n} $. Then its characteristic exponent admits the
representation 
\begin{equation*}
\psi (\mathbf{y} )=-\frac{1}{2} \left\langle A\mathbf{y}, \mathbf{y}
\right\rangle -i \langle\mathbf{b}, \mathbf{y} \rangle -\int_{\mathbb{R}^{n}
}\left( 1-e^{i \langle \mathbf{y}, \mathbf{x}\rangle}+i\langle \mathbf{y}, 
\mathbf{x}\rangle \chi _{D}(\mathbf{x})\right) \Pi (d\mathbf{x}),
\end{equation*}
where $\chi _{D}(\mathbf{x})$ is the characteristic function of $D:=\{%
\mathbf{x} \in \mathbb{R}^{n},\,\,|\mathbf{x}| \leq 1\}$, $A$ is a symmetric
nonnegative-definite $n \times n$ matrix, $\mathbf{b} \in \mathbb{R}^{n}$
and $\Pi (d\mathbf{x})$ is a measure on $\mathbb{R}^{n}$ such that 
\begin{equation}  \label{ppp}
\int_{\mathbb{R}^{n}}\min \{1,\langle\mathbf{x}, \mathbf{x}\rangle\}\Pi (%
\mathbf{x})<\infty, \,\,\Pi(\{\mathbf{0}\})=0.
\end{equation}
Hence $\widehat{\mu }\left( \mathbf{y}\right)=e^{\psi(\mathbf{y})}$.
\end{theorem}

The density of $\Pi $ is known as the \textit{L\'{e}vy density} and $A$ is
the \textit{covariance matrix}. In particular, if $A=0$ (i.e. $%
A=(a_{j,k})_{1\leq j,k\leq n}$, $a_{j,k}=0$) then the L\'{e}vy process is a 
\textit{pure non-Gaussian process} and if $\Pi =0$ the process is \textit{%
Gaussian}.

We say that the L\'{e}vy process has\textit{\ bounded variation} if its
sample paths have bounded variation on every compact time interval. A L\'{e}%
vy process has bounded variation iff $A=0$ and 
\begin{equation*}
\int_{\mathbb{R}^{n}}\min \left\{ \left\langle \mathbf{x,x}\right\rangle
,1\right\} \Pi \left( d\mathbf{x}\right) <\infty ,\,\,\Pi \left( \left\{ 
\mathbf{0}\right\} \right) =0,
\end{equation*}%
(see, e.g., \cite{bertoin}, p.15).

The systematic exposition of the theory of L\'{e}vy processes can be found
in \cite{gs1}, \cite{gs2}, \cite{gs3}, \cite{sato}, \cite{Applebaum}, \cite%
{McKean}.




\begin{thebibliography}{99}
\bibitem{bern} Akhiezer, N. I., Lectures on approximation theory, Nayka,
Moskow, 1965.

\bibitem{Applebaum} Applebaum, D., L\'evy processes and Stochastic Calculus,
Cambridge University Press, 2009.









\bibitem{bertoin} Bertoin, J., L\'{e}vy processes. Cambridge University
Press, Cambridge.


\bibitem{bp1} Bouchaud, J-P, Potters, M., Theory of financial risk,
Cambridge University press, Cambridge, 2000.

\bibitem{bl1} Boyarchenko, S. I., Levendorskii, S. Z., Non-gaussian
Merton-Black-Scholes theory, World Scientific, Advanced Series on
Statistical Science \& Applied Probability, Vol. 9, 2002.


\bibitem{Carr and Madan} Carr, P., Madan, Option valuation using the fast
Fourier transform, J. Comput. Finance, 1999, \textbf{2}, 61--73. 

\bibitem{carmona} Carmmona, R., Durrleman, V., Pricing and hedging spread
options, SIAM Review, 2003, \textbf{45}, 627-685.


\bibitem{Dempster and Hong} Dempster, M. A. H., Hong, S. S. G., Spread
option valuation and the fast Fourier transform, in Mathematical Finance,
Bachelier Congress 2000, Springer, Berlin, 2002, 203--220.

\bibitem{Deng} Deng, S., Stochastic models of energy commodity prices and
their applications: mean reversion with jumps and spikes. Working paper,
Georgia Institute of Technology, October 1999.

\bibitem{plancherel} Dias, 0., Fourier Transform and Plancherel Theorem,
Preprint, University of Texas at Austin, September 27, 2000, 1-34,
www.ma.utexas.edu/users/odiaz/notas/fourier.pdf



\bibitem{Duan} Duan, J. C., Pliska, S. R., Option valuation with
co-integrated asset prices. Working paper, Department of Finance, Hong Kong
University of Science and Technology, January 1999.


\bibitem{Duffy} Duffy, D. J. Finite difference methods in financial
engineering. Wiley Finance Series. John Wiley $\&$ Sons Ltd., Chichester. A
partial differential equation approach, 2006. 

\bibitem{gs1} Gihman, I. I., Skorohod, A. V., The Theory of Stochastic
Processes I, Springer-Verlag, 1974,

\bibitem{gs2} Gihman, I. I., Skorohod, A. V., The Theory of Stochastic
Processes II, Springer-Verlag, 1975.

\bibitem{gs3} Gihman, I. I., Skorohod, A. V., The Theory of Stochastic
Processes III, 1979.

\bibitem{goldberg} Goldberg, R. R., Fourier Transform. Academic Press,
London and New York, 1961. 

\bibitem{f1} Fama, E. F., The behaviour of stock market prices, Journ. of
Business, 1965, \textbf{38}, 34-105.


\bibitem{hurd} Hurd, T. R., Zhou, Z., A Fourier transform method for spread
option, arXiv:0902.3643v1 [q-fin.CP] 20 Feb 2009. 


\bibitem{Kirk} Kirk, E., Correlation in the energy markets, In \textit{%
Managing Energy Price Risk} (First edition). London: Risk Publications and
Enron, 1995, 71-78.



\bibitem{Korobov} Korobov, N. M., Trigonometric sums and their applications,
Moskov, Nauka, 1989.

\bibitem{Kuipers} Kuipers, L., Niderreiter, H., Uniform distribution of
sequences, Wiley-Interscience, New York, 1974.

\bibitem{Kushpel-Levesley} Gomes, S. M., Kushpel, A., Levesley, J., Ragozin
D. L., Interpolation on the torus using $sk$-splines with number theoretic
knots, J. of Approximation Theory, 1999, \textbf{98}, 56-71.




\bibitem{Li Deng Zhou} Li, M., Deng, S., Zhou, J., Closed-form
approximations for spread option prices and Greeks, J. Derivatives, 2008, 
\textbf{15}, 58-80.

\bibitem{Lord} Lord, R., Fang, F., Bervoets, F., Oosterlee, C. W., A fast
and accurate FFT-based method for pricing early-exercise options under
L\'evy processes, SIAM J. Sci. Comput., 2008, \textbf{30}, 1678-1705. 



\bibitem{m3} Madan, D. B., Seneta, E., The VG model for share market
returns, Journal of Business, 1990, \textbf{63}, 511-524.










\bibitem{man1} Mandelbrot, B. B., The variation of certain speculative
prices, Jorn. of Business, 1963, \textbf{36}, 394-419.


\bibitem{Margrabe} Margrabe, W. The value of an option to exchange one asset
for another, J. Finance, 1978, \textbf{33}, 177--186.


\bibitem{ms1} Mantegna, R. N., Stanley, H. E., Stochastic process with
ultraslow convergence to a gaussian: the truncated L\'{e}vy flight, Phys.
Rev. Lett. 1994, \textbf{73}, 946-2949.

\bibitem{ms2} Mantegna, R. N., Stanley, H. E., An introduction to
Econophysics. Correlation and complexity in Finance, Cambridge University
Press, Cambridge, 2000.

\bibitem{Marks} Marks, R. J., Advanced topics in Schannon sampling and
interpolation theory, Springer-Verlag, 1993. 

\bibitem{Mbanefo} Mbanefo, A. Co-movement term structure and the valuation
of crack energy spread options. In \textit{Mathematics of Derivatives
Securities}. M. A. H. Dempster and S. R. Pliska, eds. Cambridge University
Press, 89-102, 1997.

\bibitem{McKean} McKean, H. P. Stochastic integrals, The Rockefeller
University, Academic Press, New York, 1969. 

\bibitem{Jitse Niesen} Niesen, J., Wright, W. M., A Krylov subspace method
for option pricing, 2011, http://ssm.com/abstract=1799124.




\bibitem{Pilipovic} Pilipovic, D., Wengler, J., Basis for boptions. Energy
and Power risk Management, December 1998, 28-29.


\bibitem{sato} Sato, K., L\'{e}vy Processes and Infinitely Divisible
Distributions, Cambridge University Press, 1999.


\bibitem{sha1} Shannon, C. E., A Mathematical Theory of Communication, The
Bell System Technical Journal, \textbf{27}, 3, 1948, 379-423.

\bibitem{sha2} Shannon, C. E., Weaver, W., A Mathematical Theory of
Communication, Univ of Illinois Press, 1949.

\bibitem{shannon} Shannon, C. E., Papers on information theory and
cybernetics, Mir, Moscow 1963.

\bibitem{Shimko} Shimko, D. C., Options on futures spreads: hedging,
speculation, and valuation. \textit{The Journal of Futures Markets}, \textbf{%
14}, 2, 183-213.



\bibitem{Tavella} Tavella, D., Randall, C., Pricing financial instruments.
John Wiley $\&$ Sons Ltd., Chichester, 2000.



\bibitem{timan} Timan, A. F., Approximation theory of functions of real
variable, Gos. Izd., Moskov, 1960. 

\bibitem{Wilmott} Wilmott, P. Paul Wilmott on quantitative finance, Second
ed. John Wiley $\&$ Sons Ltd., Chichester, 2006. 

\end{thebibliography}
\end{document}